\pdfoutput=1
\RequirePackage{ifpdf}
\ifpdf 
\documentclass[pdftex]{sigma}
\else
\documentclass{sigma}
\fi

\numberwithin{equation}{section}

\newtheorem{Theorem}{Theorem}[section]
\newtheorem*{Theorem*}{Theorem}

\newtheorem{Proposition}[Theorem]{Proposition}
\newtheorem{Conjecture}[Theorem]{Conjecture}
 { \theoremstyle{definition}
\newtheorem{Definition}[Theorem]{Definition}

\newtheorem{Example}[Theorem]{Example}
\newtheorem{Remark}[Theorem]{Remark} }

\graphicspath{{./images/} }

\begin{document}

\newcommand{\arXivNumber}{2205.08197}

\renewcommand{\PaperNumber}{011}

\FirstPageHeading

\ShortArticleName{Refined and Generalized $\hat{Z}$ Invariants for Plumbed 3-Manifolds}

\ArticleName{Refined and Generalized $\boldsymbol{\hat{Z}}$ Invariants\\ for Plumbed 3-Manifolds}

\Author{Song Jin RI~$^{\rm ab}$}

\AuthorNameForHeading{S.J.~Ri}

\Address{$^{\rm a)}$~ICTP, Strada Costiera 11, Trieste 34151, Italy}
\EmailD{\href{mailto:sri@ictp.it}{sri@ictp.it}}
\Address{$^{\rm b)}$~SISSA, Via Bonomea 265, Trieste 34136, Italy}

\ArticleDates{Received September 05, 2022, in final form February 28, 2023; Published online March 19, 2023}

\Abstract{We introduce a two-variable refinement $\hat{Z}_a(q,t)$ of plumbed 3-manifold invariants $\hat{Z}_a(q)$, which were previously defined for weakly negative definite plumbed 3-manifolds. We also provide a number of explicit examples in which we argue the recovering process to obtain $\hat{Z}_a(q)$ from $\hat{Z}_a(q,t)$ by taking a limit $ t\rightarrow 1 $. For plumbed 3-manifolds with two high-valency vertices, we analytically compute the limit by using the explicit integer solutions of quadratic Diophantine equations in two variables. Based on numerical computations of the recovered $\hat{Z}_a(q)$ for plumbings with two high-valency vertices, we propose a conjecture that the recovered $\hat{Z}_a(q)$, if exists, is an invariant for all tree plumbed 3-manifolds. Finally, we provide a formula of the $\hat{Z}_a(q,t)$ for the connected sum of plumbed 3-manifolds in terms of those for the components.}


\Classification{57K31; 57R56; 11D09}

\section{Introduction}
It is well-known that Dehn surgery establishes a pivotal relation between links in $\mathbb{R}^3$ and closed oriented 3-manifolds \cite{lickorish, wallace}. This gives rise to an important relation between invariants of links and 3-manifolds. A typical example is the colored Jones polynomial for framed oriented links and the Witten--Reshetikhin--Turaev (WRT) invariant for closed oriented 3-manifolds \cite{reshetikhin-turaev, witten}.

Khovanov homology \cite{khovanov} is another well-known invariant for knots and links. Furthermore, it is a categorification of the Jones polynomial, that is, the Euler characteristic of Khovanov homology is the Jones polynomial. Therefore, it is natural to ask if there is an invariant of 3-manifolds not only as the counterpart of Khovanov homology but also as a categorification of the WRT invariants. This is one of the major open questions in quantum topology.

Some progress in this direction has been made in \cite{gppv, gpv}, where a physical definition of certain new invariants of 3-manifolds, often referred to as homological blocks, Gukov--Pei--Putrov--Vafa (GPPV) invariant, or as $\hat{Z}$ invariant, denoted by $ \hat{Z}_a(q) $, was formulated by the study of $3d$ $\mathcal{N}=2 $ supersymmetric theory obtained by compactification of~$6d$ $\mathcal{N}=(2,0)$ theory on a~3-mani\-fold. The $3d$ $\mathcal{N}=2 $ theory depends on the choice of a Lie group $ G $ and the case $G=\text{SU}(2)$ is mostly studied in \cite{gukov, gppv, gpv} since it corresponds to the Jones polynomial. In this paper, we will also assume $G=\text{SU}(2)$. The cases for other gauge groups, especially for $G=\text{SU}(N)$ are studied in~\cite{spark}. The invariant $ \hat{Z}_a(q) $ is a power series in $ q $ with integer coefficients and~$a$ denotes a~$ \text{Spin}^c $ structure of the 3-manifold. Moreover, a general relation between $ \hat{Z} $ invariants and WRT invariants was conjectured. The conjecture says that the limit of a certain linear combination of $ \hat{Z}_a(q) $ over all possible $ a $ as $ q $ goes to a root of unity would be equal to the WRT invariant. A rigorous mathematical definition of the invariants for general 3-manifolds is yet to be found. A concrete mathematical formula of $ \hat{Z}_a(q) $ and its invariance are given in \cite{gukov} and \cite{gppv} only for weakly negative definite plumbed 3-manifolds.

The purpose of this paper is to introduce a refined and generalized version of these $ q $-series, which is defined for a much larger class of plumbed 3-manifolds. To do this, we first construct a~refinement, the $ (q,t) $-series, by introducing a~new regulator variable $ t $, which is also an invariant of \textit{reduced} plumbed 3-manifolds. This $ (q,t) $-series has a nice property that for weakly negative definite plumbings the evaluation at $ t=1 $ is equal to the $ q $-series $ \hat{Z}_a(q) $. However, for weakly positive definite plumbings we can switch the $ (q,t) $-series into the $ q $-series in \cite{gukov} by taking $t=1$ and transforming the expansion in $ q^{-1} $ into an expansion in $ q $ as in \cite{cheng}. Also, the $ (q,t) $-series for negative definite plumbings is reminiscent of the 2-variable series introduced in \cite{akhmechet}, which is obtained by combining the lattice cohomology and the $ q $-series.

Another property of the $ (q,t) $-series is that the exponents of $ t $ are intimately related to non-negative integer solutions of quadratic Diophantine equations. Therefore, we can recover the $ q $-series from the $ (q,t) $-series by computing the limit $ t\rightarrow 1 $ even for a certain class of strongly indefinite plumbings by solving the corresponding quadratic Diophantine equations. In cases when the limit is finite, the recovered $ q $-series is conjecturally also an invariant of plumed 3-manifolds.

The organization of the paper is as follows. In Section \ref{section_plumbed_3-manifolds}, we review some known facts about the plumbed 3-manifolds and the $q$-series $\hat{Z}_a(q)$. In Section~\ref{section_(qt)_series}, we provide the formula for the~$\hat{Z}_a(q,t)$ invariants of reduced plumbed 3-manifolds, and prove that they are independent of the plumbing presentation. In Section~\ref{section_recovering_q_series}, we give an analytical study on the recovering process for plumbings with two high-valency vertices by using the solutions of quadratic Diophantine equations in two variables. Also, we propose a conjecture that the recovered $ q $-series is an invariant of arbitrary plumbed 3-manifolds. Finally, in Section \ref{section_connected_sum}, we provide a relation between the $(q,t)$-series of the disjoint union of two plumbing graphs and those of component plumbings. The Appendix contains a brief review on the solutions of quadratic Diophantine equations in two variables, the results of which are used in the main text.

\section[Plumbed 3-manifolds and q-series]{Plumbed 3-manifolds and $\boldsymbol{q}$-series}
\label{section_plumbed_3-manifolds}

In this section we review some known facts about plumbed 3-manifolds and their invariants, and we set up notational conventions.

\subsection{Plumbing graphs}\label{subsection_plumbing_graphs}

A \textit{plumbing graph} is a finite weighted graph $ \Gamma $, that is, a graph consisting of a finite number of vertices and edges together with the data of integer weights associated to vertices. In Section~\ref{section_plumbed_3-manifolds}--\ref{section_recovering_q_series}, we will assume that $ \Gamma $ is a tree, and we will consider disconnected plumbings in Section~\ref{section_connected_sum}.

Let $ V $ be the set of vertices of $ \Gamma $. For each vertex $ v\in V $, $ m_v $~denotes the weight of the vertex~$ v $, and the degree $ \deg(v) $ describes the number of edges connected to the vertex. Let $ s=|V| $ be the cardinality of the set $ V $. Then we define the symmetric $ s\times s $ matrix $ M=M(\Gamma) $, called \textit{linking matrix} of $ \Gamma $, by
\begin{equation*}
	M_{v_1,v_2}=
	\begin{cases}
		m_v,&\text{if } v_1=v_2=v,\\
		1,&\text{if }v_1, \ v_2 \text{ are connected by an edge},\\
		0,&\text{otherwise}.
	\end{cases}
	\qquad v_i\in V.
\end{equation*}

From $ \Gamma $, we can construct plumbed 3-manifolds in the following way: we first obtain the framed link $ L(\Gamma) $ in $S^3=\partial B^4$ by taking an unknot with framing $ m_v $ for each vertex $ v $ and making these unknots forming Hopf links whenever corresponding vertices are connected by an edge, e.g., see Figure \ref{fig_framed_link}.
\begin{figure}[t]
	\centering
	\includegraphics[scale=0.39]{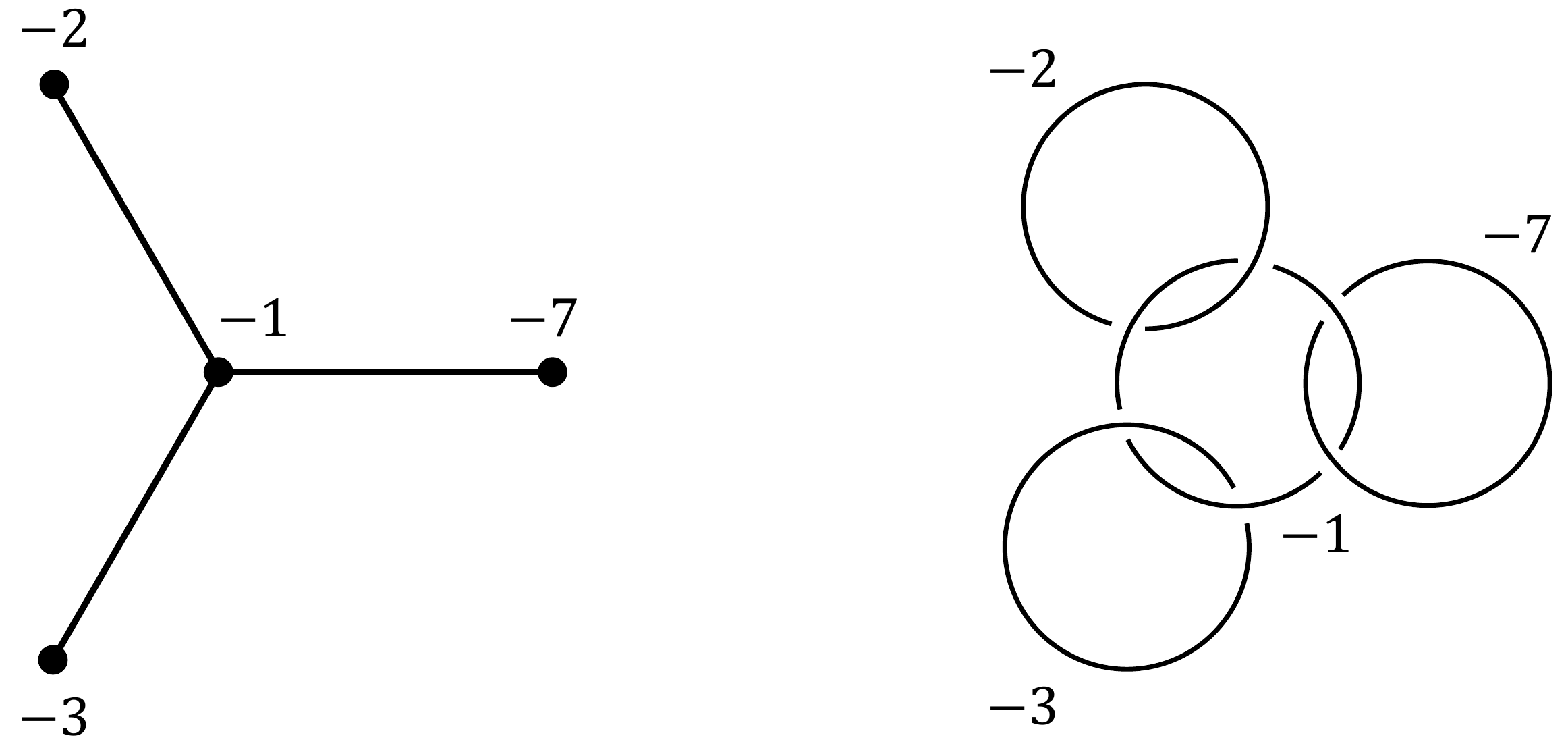}
	\caption{A plumbing graph $ \Gamma $ on the left and its associated framed link $ L(\Gamma) $ on the right.}
	\label{fig_framed_link}
\end{figure}
By attaching two-handles to $ B^4 $ along $ L(\Gamma) $, we get the 4-manifold, denoted by $ W(\Gamma) $. It can be also obtained by plumbing disk bundles over $ S^2 $ with Euler numbers~$ m_v $. Then its boundary $ Y=Y(\Gamma)=\partial W(\Gamma) $ is the closed and oriented 3-manifold obtained by Dehn surgery from $ L(\Gamma) $ whose first homology is
\[
H_1(Y)\cong \mathbb{Z}^s/M\mathbb{Z}^s.
\]

Two different plumbing graphs can represent the same homeomorphism class of $ Y(\Gamma) $. This happens if and only if they are related by a finite sequence of Neumann moves~\cite{neumann} depicted in Figure \ref{fig_neumann_moves}.
\begin{figure}[t]
	\centering
	\includegraphics[scale=0.4]{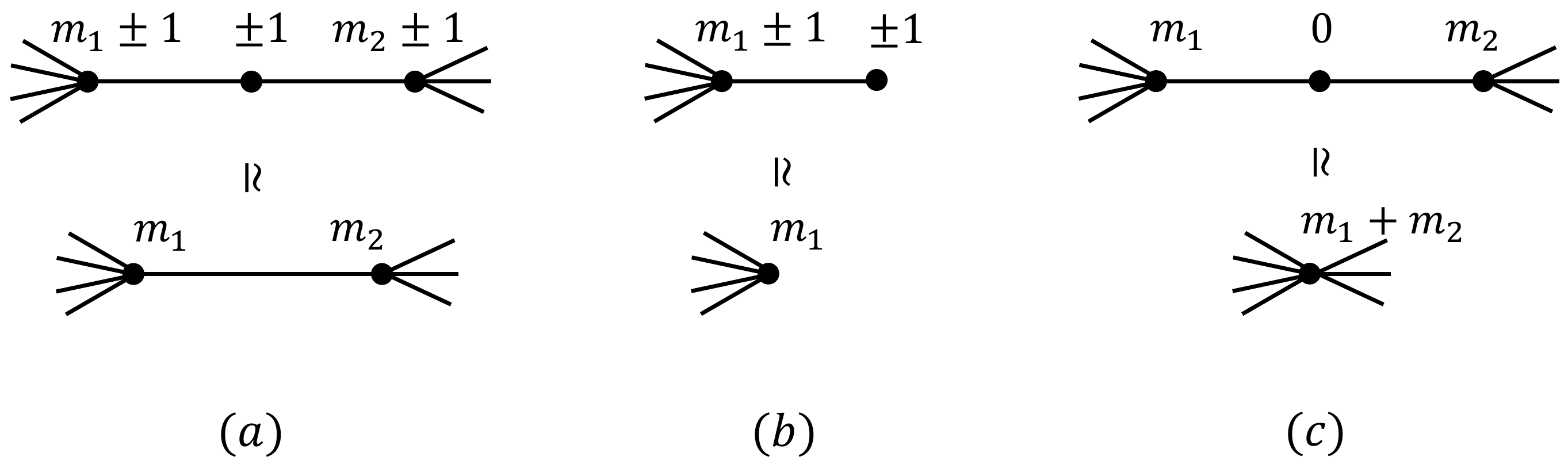}
	\caption{Neumann moves on plumbing graphs that realize homeomorphic 3-manifolds.}
	\label{fig_neumann_moves}
\end{figure}

For the purposes of this paper, we establish some terminology of plumbings that can be helpful for understanding the structures of plumbing graphs. Given a graph $ \Gamma $, we divide the set $ V $ of vertices into three disjoint subsets $ V_1$, $V_2 $ and $ V_h $ with respect to the degree of vertices as following:
\begin{gather*}
	V_1=\{v\in V\mid \deg(v)=1\}, \qquad V_2=\{v\in V\mid \deg(v)=2\}, \\
	V_h=\{v\in V\mid \deg(v)=2+p_v\geq 3, \, p_v\in \mathbb{Z}_{>0} \}.
\end{gather*}
Then we name the elements of $ V_1$, $V_2 $ and $ V_h $ by \textit{valency one}, \textit{valency two} and \textit{high-valency vertices}, respectively.

Assume that $ \Gamma $ has at least one high-valency vertex. Choose a high-valency vertex $ w\in V_h $ among them. Then $ w $ has \textit{arm}s as many as its degree $ \deg(w) $, where an arm means a part of graph, starting from a given high-valency vertex, possibly going through only a finite number of valency two vertices and ending at a valency one or another high-valency vertex. If an arm starts from $ w $ and ends at a valency one vertex $ v\in V_1 $, then we call it a \textit{branch} from $ w $ to $ v $ and denote it by $ \Gamma_{wv} $. An arm starting from a high-valency vertex $ w_1\in V_h $ and ending at another high-valency $ w_2\in V_h $ is called a \textit{bridge} between $ w_1 $ and $ w_2 $, denoted by $ \Gamma_{w_1,w_2} $. For example, two graphs in Figure~\ref{fig_neumann_moves}\,(a) have bridges between $ m_1 $ and $ m_2 $ for the bottom one or from $ m_1\pm 1 $ to $ m_2\pm 1 $ for the top, respectively, and the top graph in Figure \ref{fig_neumann_moves}\,(b) has a branch from $ m_1\pm 1 $ to $ \pm 1 $.

Since a branch $ \Gamma_{wv} $ from $ w $ to $ v $ can be thought as a linear plumbing, the following continued fraction
\begin{equation}\label{eqn_continued_fraction}
	m_w'=m_w-\cfrac{1}{u_1-\cfrac{1}{u_2-\cfrac{1}{\ddots - \cfrac{1}{m_v}}}}
\end{equation}
can give us the information of the branch. It is actually a complete invariant under the Neumann moves creating or annihilating valency two vertices along the branch. In \eqref{eqn_continued_fraction}, $ m_w $ is the weight of the starting high-valency vertex of the branch $ \Gamma_{wv} $, $ u_i $s represent the weights of valency two vertices on the linear plumbing between $ w $ and $ v $, and $ m_v $ is the weight of the ending valency one vertex. If the continued fraction $ m_w' $ of $ \Gamma_{wv} $ is an integer, the branch is called \textit{pseudo} branch because it can be annihilated or removed by a sequence of Neumann moves. For example, the top graph in Figure \ref{fig_neumann_moves}\,(b) has a pseudo-branch $\Gamma_{m_1\pm1, \pm1}$.\footnote{By abuse of notation, we often use the weights to denote the vertices where the meaning is clear in the context.} A high-valency vertex is $w\in V_h$ called a \textit{pseudo} high-valency vertex if it satisfies the following condition
\[
\deg(w) - n_w = 1, 2,
\]
where $n_w$ denotes a number of pseudo branches connected to the vertex $w$. This means that a~pseudo high-valency vertex can be turned into a valency one or two vertex once we annihilate or collapse all of its pseudo branches by using Neumann moves.

For a given bridge $ \Gamma_{w_1,w_2} $ between two high-valency vertices $ w_1,w_2\in V_h $, we also examine the continued fraction of valency two vertices laid on the bridge if they exist. The bridge is called a \textit{pseudo} bridge if it contains at least one valency two vertex and the continued fraction is zero. A pseudo bridge can also be annihilated by a sequence of Neumann moves of type (a) and (c).

We define a plumbing graph to be \textit{reduced} if there is at least one vertex with degree greater than two and there does not exist a pseudo high-valency vertex. One can change any plumbing tree into a reduced plumbing by removing pseudo high-valency vertices by using Neumann moves. Note that removing pseudo high-valency vertex could create another pseudo high-valency vertex, hence it is needed to keep removing all pseudo high-valency vertices as they appear until there is no pseudo high-valency vertex anymore. We also notice that a plumbing without any high-valency vertex represents a lens space.

\subsection[Identification of Spin\^{}c structures]{Identification of $\boldsymbol{\text{Spin}^c}$ structures}

Let us review briefly the identification of $ \text{Spin}^c $ structures on plumbed 3-manifolds in terms of plumbing data. The affine space $ \text{Spin}^c(Y) $ in the case of $ Y=Y(\Gamma) $ for a plumbing tree $ \Gamma $ has already been studied in the Heegard--Floer homology literature \cite{heegard-floer}, and a slightly different description has been used in order to construct an invariant $ \hat{Z}_a(q) $ of the plumbed 3-manifold equipped with the $\text{Spin}^c$ structure $ a $ in~\cite{gukov}. Here we recall $ \text{Spin}^c $ structures on plumbed 3-mani\-folds $ Y(\Gamma) $ by following \cite{gukov}.

The $ \text{Spin}^c $ structures on $ Y $ can be inherited by those on the 4-manifold $ W=W(\Gamma) $ with boundary $ Y $. More explicitly, a natural identification for $ \text{Spin}^c $ structures on $ W $
\begin{equation*}
	\text{Spin}^c(W)\cong 2\mathbb{Z}^s+\vec{m}
\end{equation*}
translates into a natural identification for the plumbed 3-manifold $ Y $
\begin{equation*}
	\text{Spin}^c(Y)\cong \big(2\mathbb{Z}^s+\vec{m}\big)/\big(2M\mathbb{Z}^s\big),
\end{equation*}
where $ \vec{m} $ is the vector made of the weights $ m_v $ for $ v\in V $. Note that both identifications take the conjugation of $ \text{Spin}^c $ structures to the involution $ a\leftrightarrow -a $ on the right-hand side.

Another natural identification
\begin{equation}\label{eqn_spinc_identification_Y}
	\text{Spin}^c(Y)\cong \big(2\mathbb{Z}^s+\vec{\delta}\,\big)/\big(2M\mathbb{Z}^s\big) \cong 2\operatorname{Coker}M+\vec{\delta},
\end{equation}
can be obtained by using the map
\[
\begin{split}
	\phi\colon \ \big(2\mathbb{Z}^s+\vec{m}\big)/\big(2M\mathbb{Z}^s\big) & \xrightarrow{\cong} \big(\big(2\mathbb{Z}^s+\vec{\delta}\,\big)/\big(2M\mathbb{Z}^s\big)\big),\\
	\big[\vec{\ell}\,\,\big] & \rightarrow \big[\vec{\ell}-M\vec{u}\big],
\end{split}
\]
where $ \vec{\delta} $ is the vector of the degrees of the vertices in $ \Gamma $, and $ \vec{u}=(1,1,\dots,1) $.

\subsection[The q-series]{The $\boldsymbol{q}$-series}

The $ q $-series has been firstly proposed in \cite{gppv} aimed to categorify the WRT invariant. It was originally defined for plumbed 3-manifolds coming from negative definite plumbing graphs. For a later convenience, we recall the formula of the $ q $-series $ \hat{Z}_a(q) $ for weakly negative definite plumbed 3-manifolds following \cite{gukov} and \cite{gppv}.

Let $ \Gamma $ be a plumbing graph satisfying the weakly negative definite condition, that is, the linking matrix $ M=M(\Gamma) $ is invertible and $ M^{-1} $ is negative definite on the subspace of $\mathbb{Z}^s$ spanned by the high-valency vertices. Using identifications~\eqref{eqn_spinc_identification_Y}, fix a representative $ \vec{a}\in 2\mathbb{Z}^s+\vec{\delta} $ of a class
\[
a\in \big(2\mathbb{Z}^s+\vec{\delta}\,\big)/\big(2M\mathbb{Z}^s\big).
\]
Then the $ q $-series $ \hat{Z}_a(q) $ is defined by
\begin{equation}\label{eqn_defn_q_series}
	\hat{Z}_a(q)=(-1)^{\pi}  q^{\frac{3\sigma-\sum_v m_v}{4}} \cdot \text{CT}_{\vec{z}}\left\{ \mathcal{D}(\vec{z}) \cdot \Theta_a^{M}(\vec{z})\right\},
\end{equation}
where
\begin{equation}\label{eqn_defn_discriminant_q_series}
	\mathcal{D}(\vec{z})=\prod_{v\in V} \left(z_v-\frac{1}{z_v}\right)^{2-\deg(v)}
\end{equation}
and
\begin{equation}\label{eqn_defn_Theta_function_q_series}
	\Theta_a^{M}(\vec{z})=\sum_{\vec{\ell}\in 2M\mathbb{Z}^s+\vec{a}} q^{-\frac{(\vec{\ell},M^{-1}\vec{\ell})}{4}} \prod_{v\in V} z_v^{\ell_v}.
\end{equation}
In \eqref{eqn_defn_q_series}, $ \text{CT}_{\vec{z}}$ is the operation of taking the constant terms of the formal power series in $z_v$. Also, for the factors in front of the operation $ \text{CT}_{\vec{z}}$, $ \pi=\pi(M) $ denotes the number of positive eigenvalues of the linking matrix $ M $ and $ \sigma=\sigma(M) $ is the signature of $ M $, i.e., $ \sigma=2\pi-s $.

Note that all rational functions in \eqref{eqn_defn_discriminant_q_series} should be understood as the symmetric expansion, that is, the average of the expansion as $ z_v \rightarrow 0 $ and $ z_v\rightarrow\infty $. For example, given a high-valency vertex $ w\in V_h $ with $ \deg(w)=2+p_w $, we have
\begin{align}
\left(z_w-\frac{1}{z_w}\right)^{-p_w}&=\frac{1}{2}\left.\left(z_w-\frac{1}{z_w}\right)^{-p_w}\right|_{|z_w|<1}+\frac{1}{2}\left.\left(z_w-\frac{1}{z_w}\right)^{-p_w}\right|_{|z_w|>1} \nonumber\\
		&=\frac{1}{2}\left[\left(-\sum_{i=0}^{\infty}z_w^{2i+1}\right)^{p_w}+\left(\sum_{i=0}^{\infty}z_w^{-(2i+1)}\right)^{p_w}\right]\nonumber\\
		&=\frac{1}{2}\sum_{r_w=0}^{\infty}A(r_w,p_w)\left[(-1)^{p_w}z_w^{2r_w+p_w}+z_w^{-2r_w-p_w)}\right],\label{eqn_symmetric_expansion}
	\end{align}
where $ A(r,p) $ is a binomial coefficient
\begin{equation*}
	A(r, p) = \binom{r}{p} = \frac{(r+1)(r+2)\cdots(r+p-1)}{(p-1)!}.
\end{equation*}

\begin{Remark}
	The weakly negative definite condition should be imposed to ensure that $ \hat{Z}_a(q) $ is well-defined. In fact, together with this condition, \eqref{eqn_defn_Theta_function_q_series} implies that the exponents of $ q $ in $ \hat{Z}_a(q) $ has a lower bound, and that only finitely many terms can contribute to the same exponent of $ q $.
\end{Remark}

\begin{Remark}
	The condition that $M$ is invertible means the corresponding plumbed 3-manifold is a rational homology sphere. We will mostly be interested in the case where $M$ is invertible. The $q$-series for plumbings satisfying $\det M=0$ was dealt in \cite{CGP}.
\end{Remark}

\section[The (q,t)-series invariants]{The $\boldsymbol{(q,t)}$-series invariants}\label{section_(qt)_series}

In this section we construct the formula of the $ (q,t) $-series $ \hat{Z}_a(q,t) $ for a certain class of plumbed 3-manifolds by introducing a regulator $ t $ into the $ q $-series $ \hat{Z}_a(q) $, previously defined only for weakly negative definite plumbings.

\subsection{Definition}
Let $ \Gamma $ be a reduced plumbing graph, that is, a tree plumbing that has at least one high-valency vertex, but does not have any pseudo high-valency vertex. We assume that its linking matrix $ M=M(\Gamma) $ is invertible. We keep our notations as in Section \ref{section_plumbed_3-manifolds}.

\begin{Definition}
	The $(q,t)$-series $\hat{Z}_a(q,t)$ for a reduced plumbing, equipped with some $\text{Spin}^c$ structure $ \vec{a}\in 2\mathbb{Z}^s+\vec{\delta} $, is defined by
	\begin{equation}\label{eqn_defn_(qt)_series}
		\hat{Z}_a(q,t)=(-1)^{\pi} q^{\frac{3\sigma-\sum_v m_v}{4}}\cdot \text{CT}_{\vec{z}}\left\{\frac{1}{{|\mathcal{S}|}}\sum_{\xi \in \mathcal{S}}\prod_{w\in V_h} \left.\mathcal{D}_w(\vec{z},t)\right|_{\xi_w} \cdot \Theta_a^{M}(\vec{z})\right\}.
	\end{equation}
\end{Definition}

Let us elaborate on the various elements of this formula.

For the factors in front of the $ \text{CT}_{\vec{z}} $ operation, the $ \pi=\pi(M) $ and $ \sigma $ denote the number of positive eigenvalues and the signature of the linking matrix $ M $, respectively. The theta function is the same as the one in \eqref{eqn_defn_Theta_function_q_series} for the $ q $-series, and $ \text{CT}_{\vec{z}} $ denotes the operation of taking the constant term of the Laurent series in $ z_v\in V $.

The $ \mathcal{S} $ is a set of vectors $ \xi\in \{\pm1\}^{V_h} $ that satisfy the following property
\[
\xi_{w_1}\xi_{w_2}=(-1)^{\pi(\Gamma_{w_1,w_2})+1} \quad \text{for each pseudo bridge } \Gamma_{w_1,w_2} \text{ in } \Gamma,
\]
where $\xi_{w}$ denotes the element in a vector $\xi$ corresponding to the high-valency vertex $w\in V_h$, and $\pi(\Gamma_{w_1,w_2})$ is the number of positive eigenvalues of the linking matrix associated to $\Gamma_{w_1,w_2}$. One can easily see that the set $ \mathcal{S} $ for a given $ \Gamma $ is well-defined as long as $ \Gamma $ is a tree. Moreover, the cardinality $ |\mathcal{S}| $ of the set $ \mathcal{S} $ is equal to 2 to the power of the number of high-valency vertices minus the number of pseudo bridges in $ \Gamma $. A vector $ \xi\in \mathcal{S} $ is called a \textit{chamber} for $ \Gamma $.

For a given chamber $ \xi $, the discriminant function at this chamber is defined by
\begin{equation}\label{eqn_defn_Discriminant_(qt)_series}
	\left.\mathcal{D}_w(\vec{z},t)\right|_{\xi_w}:=
	\begin{cases}
		\displaystyle{(-1)^{p_w}\sum_{r_w=0}^{\infty}A(r_w,p_w)t^{2r_w+p_w} z_w^{2r_w+p_w}\prod_{v\in I_w}\mathcal{D}_v^{\varphi(v)}(z_v,t)}&\text{for} \ \xi_w=+1,\\
		\displaystyle{\sum_{r_w=0}^{\infty}A(r_w,p_w)t^{2r_w+p_w}z_w^{-2r_w-p_w}\prod_{v\in I_w}\mathcal{D}_v^{-\varphi(v)}(z_v,t)}&\text{for} \ \xi_w=-1.
	\end{cases}
\end{equation}
Here $ I_w $, associated to each $ w\in V_h $, is the set of valency one vertices that are end-points of all branches starting from $ w $. And the discriminant $ \mathcal{D}_v^{\varphi(v)} $ of the valency one vertex $ v\in I_w $ depends on $ \varphi(v) $ by
\begin{equation*}
	\mathcal{D}_v^{\varphi(v)}:=
	\begin{cases}
		z_v t-{1}/{z_vt}&\text{for} \ \varphi(v)=+1,\\
		{z_v}/{t}-{t}/{z_v}&\text{for} \ \varphi(v)=-1,\\
		z_v-{1}/{z_v}&\text{for} \ \varphi(v)=0.
	\end{cases}
\end{equation*}
The function $ \varphi(v) $ returns a sign, determined by the branch $ \Gamma_{wv} $, as follows: if the branch $ \Gamma_{wv} $ from $ w $ to $ v $ is not a pseudo-branch, then $ \varphi(v)=0 $. If the branch $ \Gamma_{wv} $ is a pseudo-branch, by definition, the weights of the vertices laid on the branch should satisfy
\[
m_w'=m_w-\cfrac{1}{u_1-\cfrac{1}{u_2-\cfrac{1}{\ddots - \cfrac{1}{m_v}}}},
\]
for some finite integer $ m_w' $. Then we define
\[
\varphi(v)=(-1)^{\pi(\Gamma_{wv})-\pi(\Gamma_{m_w'})},
\]
where $ \Gamma_{m_w'} $ is the graph that consists of the single vertex weighted by $ m_w' $. Recall that the notation $ \pi(\Gamma_{wv}) $ denotes the number of positive eigenvalues of the linking matrix corresponding to $ \Gamma_{wv} $, so it is immediate to see $ (-1)^{\pi(\Gamma_{m_w'})}=\text{sgn}(m_w') $.

For completeness, we define $ \hat{Z}_a(q,t)=\hat{Z}_a(q) $ if there is no high-valency vertex in $ \Gamma $. The $ (q,t) $-series is well-defined even when $ \Gamma $ does not satisfy the weakly negative (or positive) definite because there are only finitely many contributions to any monomial $ t^{n_t} q^{n_q} $.
\begin{Remark}\label{remark_main_defn}
	Comparing \eqref{eqn_defn_Discriminant_(qt)_series} with \eqref{eqn_symmetric_expansion}, one can see that the new variable $ t $ is introduced as the regulator that appears in the standard $ \zeta $-regularization process. Moreover, if $ \Gamma $ is weakly negative definite, then we can recover $ \hat{Z}_a(q) $ from $ \hat{Z}_a(q,t) $ by taking $ t=1 $.
\end{Remark}

\begin{Remark}
	In \cite{gukov}, two-variable series $F_K(x,q)$ for knot complements is introduced as the analogue of the invariants $\hat{Z}_a(q)$. It is also possible to construct the $t$-deformation of $F_K$ as the analogue of $\hat{Z}_a(q,t)$ for knot complements in the same way. We note that another version of $t$-deformation and $a$-deformation has been discussed in \cite{tdeform}.
\end{Remark}

\subsection{Invariance}
For the $ (q,t) $-series $ \hat{Z}_a(q,t) $ to be an invariant of the plumbed 3-manifold $ Y $ equipped with the $ \text{Spin}^c $ structure $ a $, it should be independent on the presentation of $ Y $ as a plumbing, that is, it has to be unchanged under the Neumann moves in Figure \ref{fig_neumann_moves}.

\begin{Proposition}\label{proposition_relation_reduced_plumbings}
If two reduced plumbings represent a same $3$-manifold, then there exists a~sequence of Neumann moves such that none of those Neumann moves creates any pseudo high-valency vertices.
\end{Proposition}
\begin{proof}
Since two reduced plumbings $\Gamma$ and $\Gamma'$ realize the same 3-manifold, there exists a~sequence of plumbings $G=\{\Gamma_0=\Gamma, \Gamma_1, \Gamma_2, \dots, \Gamma_{n-1}, \Gamma_n=\Gamma'\}$ such that any adjacent two plumbings $\Gamma_{i-1}$ and $\Gamma_i$ are related by a~Neumann move depicted in Figure \ref{fig_neumann_moves} for $i=1, \dots, n$. It is important to note that in this proof we consider a Neumann move as an action applied to a certain vertex on a branch or on a bridge, instead of regarding it as a transformation from a~plumbing graph to another.
	
As a first step, we pick up a subsequence of Neumann moves such that the first move in the subsequence creates a pseudo branch or a pseudo bridge, the last move deletes the pseudo branch or bridge, and all the moves in between are actions on the branch or bridge. It is clear that the subsequence is redundant, therefore, we are going to remove all such subsequences. Notice that another subsequence with this property can be appeared after removing a subsequence, so we keep removing subsequences until they will not appear anymore. We also note that the remaining sequences of Neumann moves are still well-defined since all the moves in this subsequence are limited on a redundant branch or bridge.
	
If the relevant plumbings are all reduced after removing subsequences, then the proof is done. Therefore, without loss of generality, let us assume that there exists at least one non-reduced plumbing in the sequence $G$. For simplicity, suppose that all non-reduced plumbings in the sequence $G$ contain only one pseudo high-valency vertex, because the cases with more than one pseudo high-valency vertex can be extended in a similar way. Among non-reduced plumbings, choose the non-reduced plumbings $\Gamma_j$ and $\Gamma_k$ with the smallest index $j$ and the largest one $k$. This means that $\Gamma_{j-1}$ and $\Gamma_{k+1}$ are reduced plumbings. Furthermore, since we have removed all redundant subsequences of Neumann moves, there are only two possible ways to create the pseudo-high valency vertex in $\Gamma_j$. The first is that the pseudo high-valency vertex in $\Gamma_j$ is created by adding a pseudo branch to a valency two vertex on a pseudo bridge in $\Gamma_{j-1}$. The other is that a high-valency vertex with degree more than 3 in $\Gamma_{j-1}$ is split into two high-valency vertex connected by a pseudo bridge, one of which is the pseudo high-valency vertex.
	
	Since two cases are related to a pseudo bridge, the key is to collapse the pseudo bridge by using Neumann moves. More explicitly, we insert a plumbing $\Gamma'_{j-1}$ between $\Gamma_{j-1}$ and $\Gamma_j$ where $\Gamma'_{j-1}$ is obtained by collapsing the pseudo bridge on which the pseudo high-valency vertex is laid in $\Gamma_{j}$, only if $\Gamma_{j-1}$ has the corresponding pseudo bridge. Then we get $\Gamma'_i$ from $\Gamma_i$ for $i=j, j+1,\dots, k$ by collapsing the relevant pseudo bridge. We also insert a plumbing $\Gamma'_{k+1}$ between $\Gamma'_k$ and $\Gamma_{k+1}$ by rebuilding a pseudo bridge as same as the one in $\Gamma_{k+1}$ if it exists in~$\Gamma_{k+1}$.
	
Then the sequence $\{\Gamma_0,\dots, \Gamma_{j-1}, \Gamma'_{j-1}, \Gamma'_{j}, \dots,
\Gamma'_k, \Gamma'_{k+1}, \Gamma_{k+1}, \dots, \Gamma_n\}$ satisfies the statement of the proposition. Notice that $\Gamma'_{j-1}$ and $\Gamma'_{k+1}$ might not appear in the sequence case by case.
\end{proof}

\begin{Theorem}\label{theorem_invariance}
The series $ \hat{Z}_a(q,t) $ defined in \eqref{eqn_defn_(qt)_series} is an invariant for the reduced plumbed $3$-manifolds.
\end{Theorem}
\begin{proof}
Since $ \hat{Z}_a(q,t)$ is equal to $\hat{Z}_a(q) $ for the plumbings without high-valency vertices by definition, we can assume that $ V_h(\Gamma ) $ is not an empty set. We use a prime to distinguish the quantities associated to the top graphs in Figure~\ref{fig_neumann_moves}. For example, $ M $ is the linking matrix for the bottom graph $ \Gamma $, and $ M' $ denotes the one for the top graph~$ \Gamma' $. According to the result of Proposition~\ref{proposition_relation_reduced_plumbings}, it is enough to consider Neumann moves that does not create a pseudo high-valency vertex.
	
Consider the move (a) in Figure $ \ref{fig_neumann_moves} $, with the signs on top being~$ -1 $. By linear algebra, it is immediate to see that we have $ \sigma'=\sigma-1 $, $ \pi'=\pi $, hence the quantity~$ 3\sigma-\sum_v m_v $ does not change. Furthermore, the factors in front of the $ \text{CT}_{\vec{z}} $ operation in~\eqref{eqn_defn_(qt)_series} does not change either.
	
For the bottom graph $ \Gamma $, let us write a vector $ \vec{\ell}\in \mathbb{Z}^s $ as a concatenation
	\[
	\vec{\ell}=\big(\vec{\ell}_1,\vec{\ell}_2\big),
	\]
	such that $ \vec{\ell}_1 $ describes the left part of the graph including the vertex $ m_1 $ and $ \vec{\ell}_2 $ describes the right part. Then we can construct a vector $ \vec{\ell}' $ for the top graph $ \Gamma' $ by
	\[
	\vec{\ell}'=\big(\vec{\ell}_1,0,\vec{\ell}_2\big)\in \mathbb{Z}^{s+1},
	\]
	that satisfies
	\begin{equation}\label{eqn_ells_relation_move_A}
		\big(\vec{\ell},M^{-1}\vec{\ell}\,\big)=\big(\vec{\ell}',(M')^{-1}\vec{\ell}'\big).
	\end{equation}
	Moreover, it is shown in \cite{gukov} that $ \vec{\ell}' $ has the property $ \vec{\ell}'\in 2M'\mathbb{Z}^{s+1}+\vec{a}' $ whenever $ \vec{\ell}\in 2M\mathbb{Z}^s+\vec{a} $, where $ \vec{a}' $ is a representative of the $ \text{Spin}^c $ structure $ a' $ for $ Y(\Gamma') $ that is the counterpart of $ a $ through the isomorphism $ Y(\Gamma)\cong Y(\Gamma') $.
	
	Taking these into account, when the move happens inside a branch $ \Gamma_{wv} $ from a high-valency vertex $ w\in V_h $ to a valency one vertex $ v\in I_w $, observe that there is no change for $ \varphi(v) $ whether the branch is pseudo or not. This implies that the discriminant function does not change, hence we obtain the same result.
	
	For the case when the move happens inside a bridge between two high-valency vertices, two discriminants for $ \Gamma' $ and $ \Gamma $ have the same structure, therefore they yield the same $ \hat{Z}_a(q,t) $.

	Let us consider now the case of the move (a) with the sign $ +1 $ for the top graph in a similar way. As before, the quantity $ 3\sigma-\sum_v m_v $ is still unchanged, but we have $ \pi'=\pi+1 $, so $ (-1)^{\pi} $ switches the sign. Given $ \vec{\ell}=\big(\vec{\ell}_1,\vec{\ell}_2\big)$ for the bottom graph, a vector for the top graph
	\[
	\vec{\ell}'=\big(\vec{\ell}_1,0,-\vec{\ell}_2\big)
	\]
	satisfies \eqref{eqn_ells_relation_move_A}.
	
	For the move on a bridge that is not pseudo, the discriminant does not change, but due to the change of sign in $ \vec{\ell}_2 $, the theta function changes as if we do the substitutions $ z_v\rightarrow z_v^{-1} $ for all the vertices $ v $ corresponding to $ \vec{\ell}_2 $. From this, we obtain extra $ -1 $ sign, that cancels out the disagreement of signs by $ (-1)^{\pi} $. Note that the powers of $ t $ remains the same by the substitutions $ z_v\rightarrow z_v^{-1} $. If the move happens on a pseudo-bridge, the proof is similar, but the only difference is that due to the minus sign in $ -\vec{\ell}_2 $ for $ \vec{\ell}' $ the choice of the chamber for the top graph should be opposite than that for the bottom, because the move increases by 1 the number of the positive eigenvalues of the linking matrix associated to the bridge.
	
	If the move appears on a non-pseudo branch $ \Gamma_{wv} $, then the mechanism is same as one for the move on a bridge, so we get the same result. However, if the move appears on a pseudo-branch~$\Gamma_{wv} $, then we need to care about the change of $ \varphi(v) $ for the discriminant function because the move (a) with the sign $ +1 $ increases the number $ \pi(\Gamma_{wv}) $ of positive eigenvalues of the linking matrix associated to the branch by~$ 1 $. In this case, $ v $ is the only vertex corresponding to~$ \vec{\ell}_2 $, hence the change of the sign for $ \ell_v $ produces the extra sign, but it gives the same powers of~$ t $ because of the change of $ \varphi(v) $. Also, the extra sign remedies the change of~$ (-1)^{\pi} $.
	
	Now, we consider the move (b) when the sign for the blow-up vertex is $ -1 $. We have $ \pi'=\pi $, $ \sigma'=\sigma-1 $, and the quantity $ 3\sigma-\sum_v m_v $ for the top graph is 1 lower than that for the bottom one. This yields an extra factor $ q^{-1/4} $ for the top graph.
	
	For the bottom graph $ \Gamma $, we write vectors as
	\[
	\vec{\ell}=\big(\vec{\ell}_2,\ell_1\big),
	\]
	where $ \ell_1 $ denotes the vertex with weight $ m_1 $. Then we define corresponding vectors for the top graph $ \Gamma' $
	\[
	\vec{\ell}_{\pm}'=\big(\vec{\ell}_2,\ell_1\pm 1,\mp 1\big).
	\]
	By simple linear algebra, one can observe that
	\begin{equation}\label{eqn_ells_relation_move_B}
		\big(\vec{\ell},M^{-1}\vec{\ell}\,\big)=\big(\vec{\ell}'_{\pm}, (M')^{-1}\vec{\ell'}_{\pm}\big)+1.
	\end{equation}
	The extra factor $ +1 $ in \eqref{eqn_ells_relation_move_B} gives rise to $ q^{1/4} $ for the top graph that cancels with $ q^{-1/4} $ coming from the factor in front of the operation $ \text{CT}_{\vec{z}} $.
	
	Let us consider the case where the vertex $ w $ decorated with $ m_1 $ is a high-valency vertex in the bottom graph. We assume that the high-valency vertex does not have a pseudo bridge. One can follow the similar argument for the case when it has a pseudo bridge.
	
	The corresponding part of the discriminant function for $ \Gamma $ is given by
	\[
	\frac{1}{2}\sum_{r=0}^{\infty} A(r,p) t^{2r+p} \big[(-1)^{p}z_1^{2r+p}\mathcal{D}^{+}+z_1^{-2r-p}\mathcal{D}^{-}\big],
	\]
	where $ \mathcal{D}^{\pm} $ are the contributions from vertices in $ I_w $ corresponding to $ \vec{\ell}_2 $, and $ p $ is determined by $ \deg(w)=2+p $. From the role of the operation $ \text{CT}_{\vec{z}} $ in \eqref{eqn_defn_(qt)_series}, it follows that $ \ell_1 $ has values in the form of $ \ell_1=\pm(2r+p) $ for non-negative integers $ r $. Without loss of generality, suppose that $ \ell_1=2r+p $. Then the vectors $ \vec{\ell} = \big(\vec{\ell}_2,2r+p\big) $ pick up monomials in $ t $ given by
	\[
	\frac{1}{2}A(r,p) t^{2r+p}.
	\]
	
	For the top graph, the discriminant function has the following portion
	\[
	\frac{1}{2}\sum_{s=0}A(s,p+1) t^{2s+p+1} \left[(-1)^{p+1}z_1^{2s+p+1}\left(z_0 t-\frac{1}{z_0 t}\right)\mathcal{D}^{+}+z_1^{-2s-p-1}\left(\frac{z_0}{t}-\frac{t}{z_0}\right)\mathcal{D}^{-}\right],
	\]
	where $ z_0 $ is the variable for the newly introduced blow-up vertex in $ \Gamma' $. Corresponding to $ \vec{\ell}=\big(\vec{\ell}_2,2r+p\big) $, we have vectors of the form
	\[
	\vec{\ell}'_{\pm}=\big(\vec{\ell}_2,2r+p\pm 1, \mp1\big).
	\]
	Therefore, the monomials chosen by those vectors are
	\[
	-\frac{1}{2}A(r-1,p+1)t^{2r+p-1}\cdot t+\frac{1}{2}A(r,p+1)t^{2r+p+1}\cdot\frac{1}{t}=\frac{1}{2}A(r,p)t^{2r+p}.
	\]
	This means that we get the same answer for $ \hat{Z}_a(q,t) $.
	
	For the case when the move applies to a valency one vertex $ v_1\in I_w $ for some $ w\in V_h $, it is essential to check that the discriminant for $ \Gamma' $ can be obtained from the one for $ \Gamma $ by the substitution $ z_1\rightarrow z_0 $ since we have $ \varphi(v_1)=\varphi(v_0) $. Then, vectors for $ \Gamma $ are described by
	\[
	\vec{\ell}=\big(\vec{\ell}_2,\pm1\big)
	\]
	and the corresponding vectors for $ \Gamma' $ are given by
	\[
	\vec{\ell}'=\big(\vec{\ell}_2,0,\pm1\big).
	\]
	Putting all together, we get the same result.
	
	Move (b) with the sign $ +1 $ is similar, but we use the vectors of the form
	\[
	\vec{\ell}'=\big(\vec{\ell}_2,\ell_1\pm1,\pm1\big).
	\]
	
	Let us move on to the move (c). Here we have $ \pi'=\pi+1 $ that yields an extra sign, and the factor $ 3\sigma-\sum_v m_v $ is unchanged. For the top graph $ \Gamma' $, we write vectors as
	\[
	\vec{\ell}'=\big(\vec{\ell}_1,\ell_1,0,\ell_2,\vec{\ell}_2\big),
	\]
	where $ \vec{\ell}_1 $ is corresponding to the left side of the graph (not including $ v_1 $) and $ \vec{\ell_2} $ is corresponding to the right side of the graph (not including $ v_2 $). From $ \vec{\ell}' $ we define a vector for the bottom graph $ \Gamma $ as
	\[
	\vec{\ell}=\big(\vec{\ell}_1,\ell_1-\ell_2,-\vec{\ell}_2\big),
	\]
	that satisfies the following equality
	\[
	\big(\vec{\ell},M^{-1}\vec{\ell}\,\big)=\big(\vec{\ell}',(M')^{-1}\vec{\ell}' \big).
	\]
	
	Let $ v_1$, $v_0$, $v_2 $ be the vertices shown in the top graph and $ v_b $ the vertex shown in the bottom one. There are four different cases depending on where those vertices are placed in the graph.
	
	First, we consider the case when $ v_1$, $v_0$, $v_2 $ are laid on a non-pseudo bridge, but at least one of $ v_1 $ and $ v_2 $ is not a high-valency vertex. Then the discriminant functions for $ \Gamma $ and $ \Gamma' $ have the same structure, and the minus sign in front of $ \vec{\ell}_2 $ in $\vec{\ell}'=\big(\vec{\ell}_1,\ell_1,0,\ell_2,-\vec{\ell}_2\big)$ produces the opposite sign for monomials in~$ t $ to the sign from $ \vec{\ell}=\big(\vec{\ell}_1,\ell_1-\ell_2,\vec{\ell}_2\big) $. This sign difference is recovered by the extra sign from $ (-1)^{\pi} $. Thus we get the same result for $ \hat{Z}_a(q,t)$.
	
If the vertices $ v_1$, $v_0$, $v_2 $ are laid on a pseudo bridge, we should put into our consideration the change of choice of chambers for the top graph since the move creates a positive eigenvalue of the linking matrix.
	
The second case is when $ v_1$, $v_0$, $v_2 $ form a bridge between $ v_1 $ and $ v_2 $, that is, $ v_1 $ and $ v_2 $ are both high-valency vertices, and $ v_0 $ is a valency two vertex. This case is just the one depicted in Figure~\ref{fig_neumann_moves}\,(c). For the bottom graph $ \Gamma $, we can assume that the vertex $ v_b $ does not have a pseudo bridge, then the corresponding portion of the discriminant function is given by
	\[
	\frac{1}{2}\sum_{r=0}^{\infty}A(r,p)t^{2r+p}\big[(-1)^p z_b^{2r+p}\mathcal{D}_1^{+}\mathcal{D}_2^{+}+z_b^{-2r-p}\mathcal{D}_1^{-}\mathcal{D}_2^{-}\big],
	\]
	where $ \mathcal{D}_1 $ and $ \mathcal{D}_2 $ denote the discriminant functions of valency one vertices in $ I_{v_b} $ corresponding to $ \vec{\ell}_1 $ and $ \vec{\ell}_2 $, respectively. Then the vectors of the form
	\[
	\vec{\ell}=\big(\vec{\ell}_1, 2r+p, \vec{\ell}_2\big)
	\]
	choose the following monomials
	\[
	\frac{1}{2}A(r,p)t^{2r+p}.
	\]
	For the top graph, since the elements $ \xi_{v_1} $ and $ \xi_{v_2} $ of a chamber $ \xi $ should be opposite, we have
	\begin{gather*}
\frac{1}{2}\left[\sum_{r_1=0}^{\infty}\right.(-1)^{p_1}A(r_1,p_1)t^{2r_1+p_1}z_1^{2r_1+p_1}\mathcal{D}_1^{+}\times\sum_{r_2=0}^{\infty}A(r_2,p_2)t^{2r_2+p_2}z_2^{-2r_2-p_2}\mathcal{D}_2^{-}\\
\qquad {}+\left.\sum_{r_1=0}^{\infty}A(r_1,p_1)t^{2r_1+p_1}z_1^{-2r_1-p_1}\mathcal{D}_1^{-}\times\sum_{r_2=0}^{\infty}(-1)^{p_2}A(r_2,p_2)t^{2r_2+p_2}(-1)^{p_2}z_2^{2r_2+p_2}\mathcal{D}_2^{+}\right],
\end{gather*}
	where $ p_1=\deg(v_1)-2 $ and $ p_2=\deg(v_2)-2 $ such that $ p=p_1+p_2 $. Associated to $ \vec{\ell}=\big(\vec{\ell}_1,2r+p,\vec{\ell}_2\big) $, we have the following vectors
	\[
	\vec{\ell}'=\big(\vec{\ell}_1,2r_1+p_1,0,-2r_2-p_2,-\vec{\ell}_2\big).
	\]
	Here the sum of $ r_1 $ and $ r_2 $ should be equal to $ r $. From such vectors, we obtain
	\[
	-\frac{1}{2}\sum_{r_1+r_2=r}A(r_1,p_1)A(r_2,p_2)t^{2r_1+2r_2+p_1+p_2}=-\frac{1}{2}A(r,p)t^{2r+p},
	\]
	where we have used $ A(r,p)=\sum_{r_1+r_2=r}A(r_1,p_1)A(r_2,p_2) $.
	
	Thirdly, let us consider when $ v_1$, $v_0$, $v_2 $ are vertices on a branch from $ v_1 $ to $ v_2 $ through $ v_0 $. Then, the discriminant function for the bottom graph $ \Gamma $ has the following part
	\[
	\frac{1}{2}\sum_{r=0}^{\infty} A(r,p)t^{2r+p} \big[(-1)^p z_b^{2r+p}\mathcal{D}^{+}+z_b^{-2r-p}\mathcal{D}^{-}\big].
	\]
	The vectors $ \vec{\ell}=\big(\vec{\ell}_1,\ell_b\big)=\big(\vec{\ell}_1,2r+p\big) $ pick up the factors of the form
	\[
	\frac{1}{2}A(r,p)t^{2r+p}.
	\]
	On the other hand, for the top graph $ \Gamma' $ we have
	\[
	\frac{1}{2}\sum_{s=0}^{\infty}A(s,p+1)t^{2s+p+1}\left[(-1)^{p+1}z_1^{2s+p+1}\left(\frac{z_2}{t}-\frac{t}{z_2}\right)\mathcal{D}^{+}+z_1^{-2s-p-1}\left(z_2t-\frac{1}{z_2t}\right)\mathcal{D}^{-}\right],
	\]
	where we have used $ \varphi(v_2)=-1 $ since the branch $ \Gamma'_{v_1 v_2} $ is a pseudo-branch and the move increases the number of positive eigenvalues by 1. Corresponding to $ \vec{\ell}=\big(\vec{\ell}_1,2r+p\big) $, we have two vectors $ \vec{\ell}'=\big(\vec{\ell}_1,2r+p+1,1\big) $ and $ \vec{\ell}'=\big(\vec{\ell}_1,2r+p-1,-1\big) $, which produce the following factors:
	\[
	\frac{1}{2}A(r,p+1)t^{2r+p+1}\cdot\frac{1}{t}+\frac{1}{2}A(r-1,p+1)t^{2r+p-1}\cdot t=-\frac{1}{2}A(r,p)t^{2r+p}.
	\]
	The minus sign in the right side of the above equation cancel with the extra sign from $ (-1)^{\pi} $. Therefore, we obtain the same result.
	
	At last, if $ v_1$, $v_0$, $v_2 $ are a part of a branch, then we use the similar approach to the third case.
\end{proof}

Now we present a simple example of the $ (q,t) $-series for a 3-manifold realized by the plumbing shown in Figure~\ref{fig_Sigma_235}. It is known as Poincar\'{e} homology sphere. It is also an example of a~Brieskorn 3-sphere, denoted by~$ \overline{\Sigma(2,3,5)} $. The overline denotes the reverse of the orientation compared to the standard one.

\begin{figure}[t]\centering
	\includegraphics[scale=0.38]{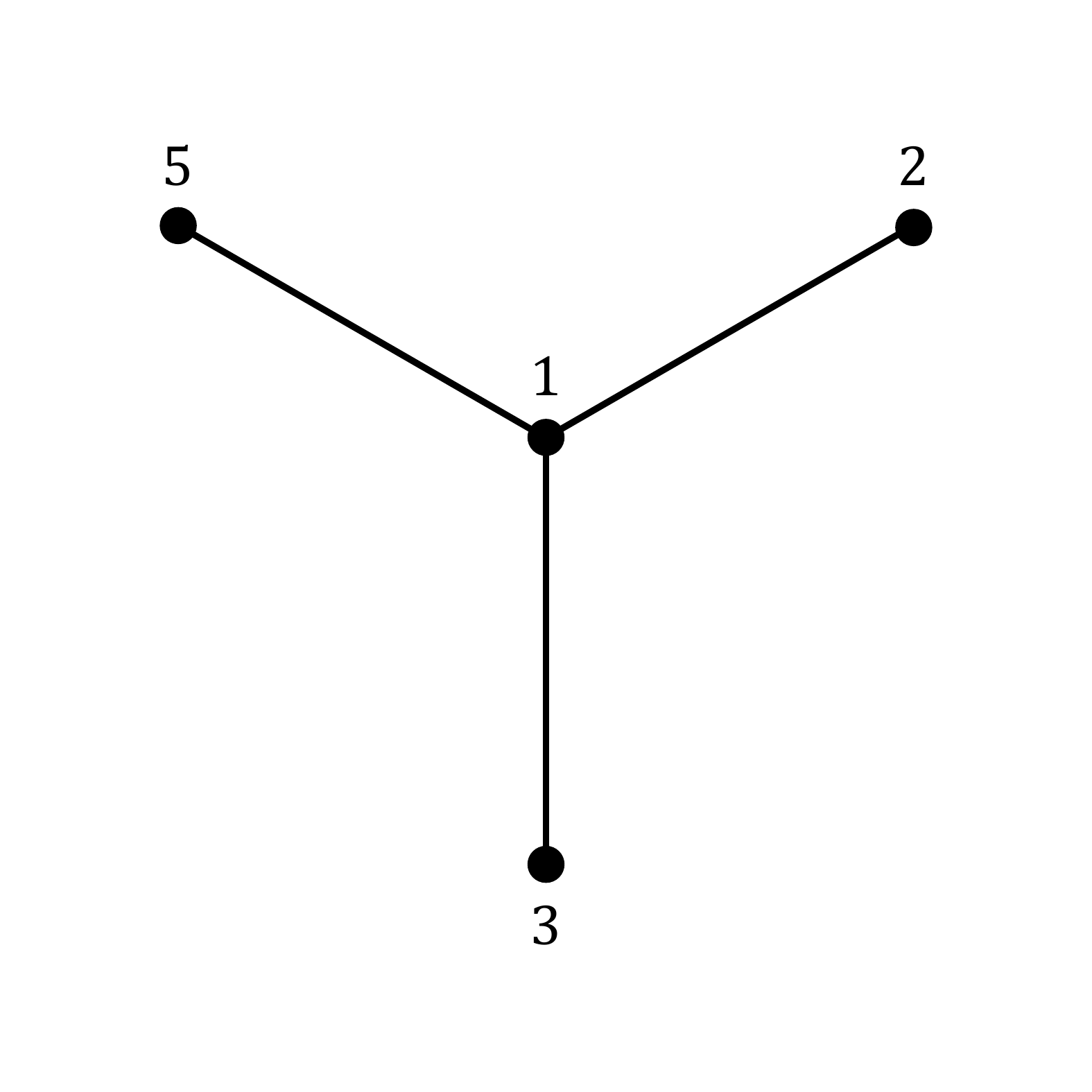}
	\caption{A plumbing graph that realizes Poincar\'{e} homology sphere $ \overline{\Sigma(2,3,5)}$.}
	\label{fig_Sigma_235}
\end{figure}

The linking matrix for the plumbing is given by
\[
M=\begin{pmatrix}
	1&1&1&1\\
	1&2&0&0\\
	1&0&3&0\\
	1&0&0&5
\end{pmatrix},
\]
and we have
\[
\pi=3, \qquad \frac{3\sigma-\sum_v m_v}{4}=-\frac{5}{4}.
\]

As $ H_1\big(\overline{\Sigma(2,3,5)},\mathbb{Z}\big)\cong 0 $, there is a unique value $ a=0 $ to consider. The formula \eqref{eqn_defn_(qt)_series} for the manifold reads as
\begin{gather*}
		\hat{Z}_0(q,t)=-\frac{1}{2}q^{-\frac{5}{4}} \text{CT}_{\vec{z}}\left\{\sum_{r=0}^{\infty}t^{2r+1}\left(z_1^{-2r-1}-z_1^{2r+1}\right)\right.\\
 \left.
\hphantom{\hat{Z}_0(q,t)=-\frac{1}{2}q^{-\frac{5}{4}} \text{CT}_{\vec{z}}}{}
 \times \left(z_2-\frac{1}{z_2}\right)\left(z_3-\frac{1}{z_3}\right)\left(z_4-\frac{1}{z_4}\right)\cdot\sum_{\vec{\ell}\in 2M\mathbb{Z}^4+\vec{\delta}}q^{-\frac{(\vec{\ell},M^{-1}\vec{\ell})}{4}}\prod_{v=1}^{4}z_v^{\ell_v}\right\}.
	\end{gather*}

Due to the operation $ \text{CT}_{\vec{z}} $, the values of $ \ell_v $ should be
\[
\ell_1=\pm(2r+1), \qquad \text{and} \qquad  \ell_v=\pm1 \quad \text{for } v=2,3,4.
\]
Observe that such vectors $ \vec{\ell} $ also satisfy the condition $ \vec{\ell}\in 2M\mathbb{Z}^4+\vec{\delta} $. Then, we obtain the $ (q,t) $-series for $ \overline{\Sigma(2,3,5)} $ as
\begin{gather}\label{eqn_(q,t)_result_Sigma235}
	\hat{Z}_0(q,t)=tq^{-\frac{3}{2}}\big(1 - {q} - q^{3} - q^{7} + t q^{8} + q^{14} + q^{20} + t^2 q^{29} - q^{31} - t^2 q^{42} - t^2 q^{57}+\cdots\big).\!\!\!
\end{gather}

Since the plumbing graph in Figure~\ref{fig_Sigma_235} has the weakly negative definite property, $ \hat{Z}_0(q,t) $ has a lower bound on the exponents of $ q $, and $ \hat{Z}_0(q,t=1) $ recovers the $ q $-series for $ \overline{\Sigma(2,3,5)} $.

\section[Recovering the q-series]{Recovering the $\boldsymbol{q}$-series}\label{section_recovering_q_series}

Let $ \Gamma $ be a reduced plumbing graph that realizes a 3-manifold $ Y=Y(\Gamma)$. By applying the formula $ \eqref{eqn_defn_(qt)_series} $, we have the invariants $ \hat{Z}_a(q,t) $ for $ Y $, which can be expressed by
\begin{equation*}
	\hat{Z}_a(q,t)=\frac{1}{{|\mathcal{S}|}}q^{\Delta_a}\sum_{\xi \in \mathcal{S}}\sum_{n\in\mathbb{Z}}C_{\xi,n}(t) q^n,
\end{equation*}
where $ \Delta_a\in\mathbb{Q} $, $ \mathcal{S} $ is the set of all possible chambers for $ \Gamma $, and $ C_{\xi,n}(t)\in \mathbb{Z}[[t]] $ is in general the formal power series in $ t $ as the coefficient of $ q^{n} $ inside a chamber $ \xi\in \mathcal{S} $.

As we have seen in the example of Section \ref{section_(qt)_series}, $ C_{\xi,n}(t) $ turns out to be a finite polynomial in~$ t $ if~$ \Gamma $ has the weakly negative (or positive) definite property. Furthermore, by simply setting $ t=1 $, we can recover the $ q $-series from the $ (q,t) $-series. Being inspired by this, it would be nice if we can recover the $q$-series by taking a limit $ t\rightarrow 1^{-} $. However, without weakly definite property, $ C_{\xi,n}(t) $ is in general infinite series and $ \lim_{t\rightarrow 1^{-}} C_{\xi,n}(t)$ might be ill-defined. This means that it would be not easy to recover the $ q $-series for such plumbings. By abuse of notation, we will denote the limit from below $ t\rightarrow 1^{-} $ by $ t\rightarrow 1 $. We notice that $ C_{\xi,n}(t) $ is convergent for $ |t|<1 $ due to the boundedness of its coefficient polynomial.

Fortunately, there are some examples of plumbings, for which it is possible to recover the $ q $-series by computing the limit $ t\rightarrow 1 $. Since the $ (q,t) $-series $ \hat{Z}_a(q,t) $ for reduced plumbings are invariant under Neumann moves, its recovered version, $\hat{Z}_a^R(q)=\lim_{t\rightarrow1} \hat{Z}_a(q,t)$ , are also invariant, therefore indeed define a topological invariant of plumbed 3-manifolds.

\begin{Remark}
	The recovered $q$-series $\hat{Z}_a^R(q)$ is exactly same as previously defined $q$-series $\hat{Z}_a(q)$ when plumbings are weakly negative definite. However, there are some examples where the recovered $q$-series are not equal to previously defined $q$-series $\hat{Z}_a(q)$, for example, weakly positive definite plumbings \cite{cheng, gukov}. Furthermore, when plumbings have strongly indefinite property, $q$-series $\hat{Z}_a(q)$ for them have not been defined, while $\hat{Z}_a^R(q)$ still exists in certain cases.
\end{Remark}

\begin{Example}
	The first example we consider is to compute the $(q,t)$-series for two reduced plumbings, shown in Figure~\ref{fig_break_weaklynd_example0}, that are equivalent by a Neumann move of type (c) in Figure~\ref{fig_neumann_moves}.
	\begin{figure}[t]
		\centering
		\includegraphics[scale=0.4]{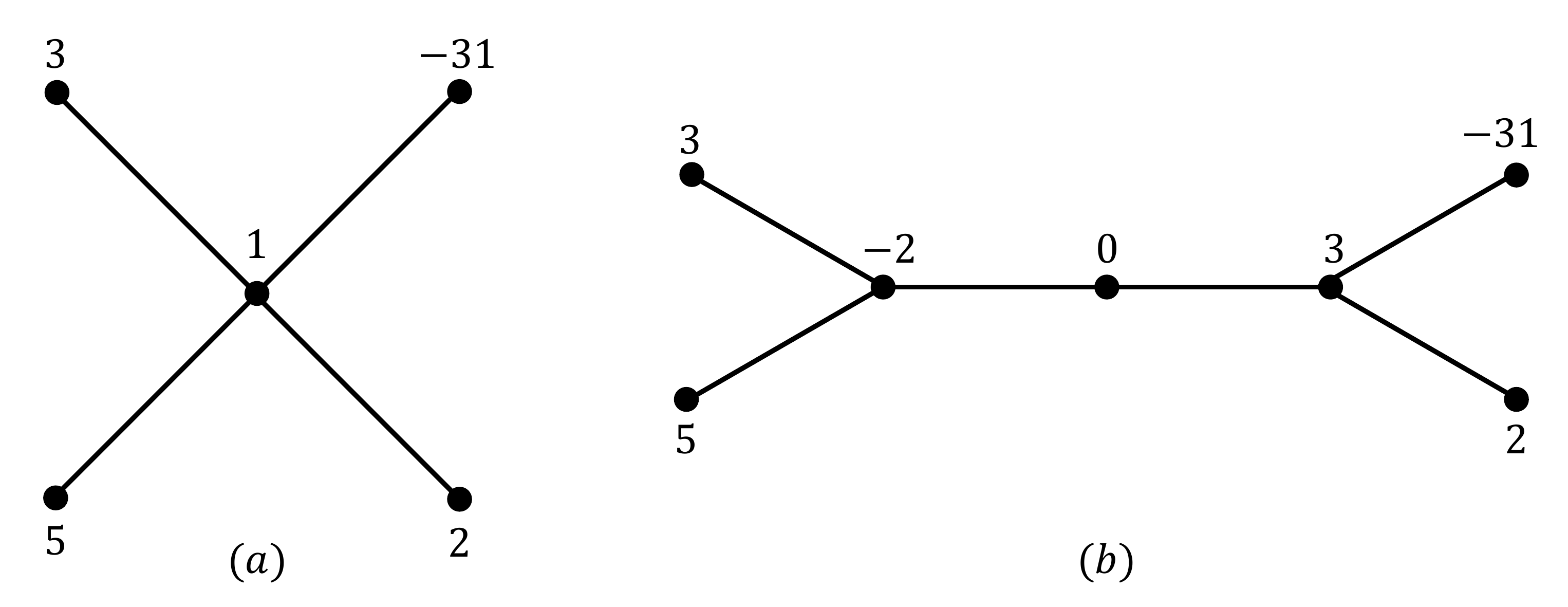}
		\caption{A Neumann move from a plumbing (a) to (b) does not preserve the weakly negative definite property.}\label{fig_break_weaklynd_example0}
	\end{figure}
	The interesting point here is that the plumbing depicted Figure~\ref{fig_break_weaklynd_example0}\,(a) has the weakly negative definite property but another plumbing in Figure \ref{fig_break_weaklynd_example0}\,(b) which is obtained by applying one Neumann move to the plumbing (a) does not have the weakly negative definite property. This means the plumbing (a) has the $q$-series by applying the formula~\eqref{eqn_defn_q_series} while the plumbing (b) does not have it. We are going to overcome such drawback by the recovering process.
	
By the formula \eqref{eqn_defn_q_series}, the $q$-series for the plumbing (a) is given by
\begin{gather}
		\label{eqn_example_4_0_q_series}
		\hat{Z}_0(q) = q^{\frac{417}{2}}\big(1-q^{29}-q^{211}+q^{252}-q^{393}+q^{442}-q^{667}+q^{726}+\cdots\big),
\end{gather}
where we notice that the determinant of the linking matrix for the plumbing is~1 and it has the unique $q$-series.
	
	The formula~\eqref{eqn_defn_(qt)_series} for the plumbing (a) reads as
	\begin{gather*}
			\hat{Z}_0(q, t) =  -\frac{1}{2}q^{\frac{23}{4}}\text{CT}_{\vec{z}}\Bigg\{\sum_{r=0}^{\infty}(r+1)t^{2r+2}\big(z_1^{2r+2}+z_1^{-2r-2}\big) \\
  \hphantom{\hat{Z}_0(q, t) =}{}
 \times\left(z_2-\frac{1}{z_2}\right)\left(z_3-\frac{1}{z_3}\right)\left(z_4-\frac{1}{z_4}\right)\left(z_5-\frac{1}{z_5}\right)\cdot \sum_{\vec{\ell}\in 2M\mathbb{Z}^5+\vec{\delta}} q^{-\frac{(\vec{\ell},M^{-1}\vec{\ell})}{4}}\prod_{v=1}^5 z_v^{\ell_v}\Bigg\},
		\end{gather*}
	which gives us the following result
	\begin{equation}
		\label{eqn_qt_result_example_4_0}
		\hat{Z}_0(q,t)=q^{\frac{417}{2}}t^2\big(1-q^{29}-q^{211}+q^{252}-q^{393}+q^{442}-q^{667}+q^{726}+\cdots\big).
	\end{equation}
	It is clear that the limit $t\rightarrow1$ of \eqref{eqn_qt_result_example_4_0} is the same as $q$-series in \eqref{eqn_example_4_0_q_series}.
	
	Moreover, one can also check that $(q,t)$-series for the plumbing (b) is equal to \eqref{eqn_qt_result_example_4_0}, which is obvious by Theorem \ref{theorem_invariance}. We note that $q$-series for the plumbing (b) does not exist because of non-weakly negative definite property, but we can associate the recovered $q$-series to it.
\end{Example}

\begin{Example}
	Let us consider a reduced plumbing depicted in Figure \ref{fig_recovering_example1}. As we will see in Section \ref{section_connected_sum}, this plumbing represents a 3-manifold which is diffeomorphic to the disjoint union of two 3-manifolds.
	\begin{figure}[t]
		\centering
		\includegraphics[scale=0.4]{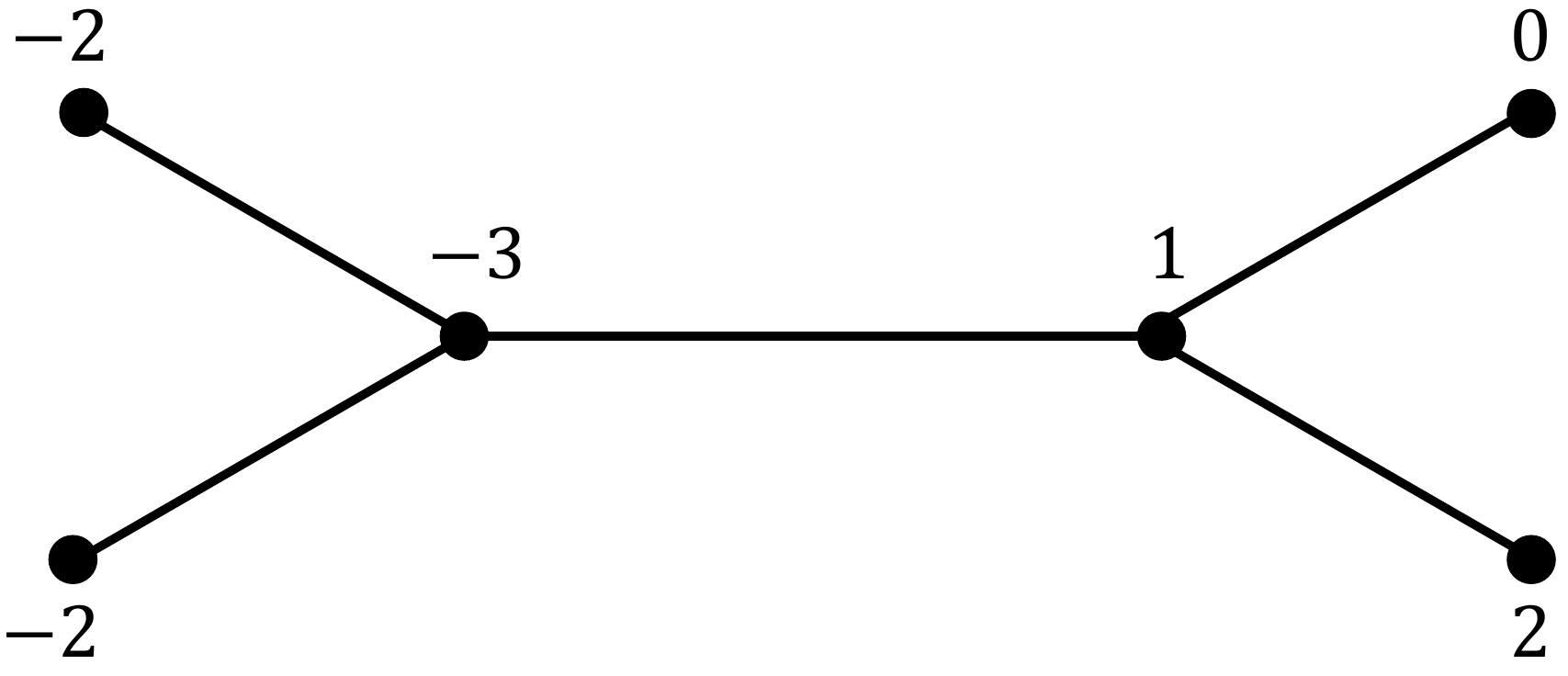}
		\caption{A plumbing graph that does not satisfy the weakly definite property.}
		\label{fig_recovering_example1}
	\end{figure}
In order to compute the $ (q,t) $-series, the first step is as usual to get the linking matrix and its inverse:
	\[
	M=\begin{pmatrix}
		-3&1&1&1&0&0\\
		1&1&0&0&1&1\\
		1&0&-2&0&0&0\\
		1&0&0&-2&0&0\\
		0&1&0&0&0&0\\
		0&1&0&0&0&2
	\end{pmatrix}, \qquad M^{-1}=\frac{1}{8}\begin{pmatrix}
		-4&0&-2&-2&4&0\\
		0&0&0&0&8&0\\
		-2&0&-5&-1&2&0\\
		-2&0&-1&-5&2&0\\
		4&8&2&2&-8&-4\\
		0&0&0&0&-4&4
	\end{pmatrix}.
	\]
	Then, we have the following quantities
	\[
	\det M=16, \qquad \pi=\pi(M)=2,\qquad \frac{3\sigma-\sum_v m_v}{4}=-\frac{1}{2}
	\]
	and the set of chambers is given by
	\[
	\mathcal{S}=\{(1,1),(1,-1),(-1,1),(-1,-1)\},
	\]
	because no pseudo bridge is contained in the plumbing. Notice that the submatrix of $ M $ corresponding to the high-valency vertices
	\[
	M_h=\begin{pmatrix}
		-1/2&0\\
		0&0
	\end{pmatrix}
	\]
	is obviously not negative or positive definite, which means that the $ q $-series $ \hat{Z}(q) $ is ill-defined.
	
Among 16 $ \text{Spin}^c $ structures, in this example for simplicity we consider only one of them, that is coming from $ (0, 0, \dots, 0)\in \operatorname{Coker} M $ by the following identification
	\[
	\text{Spin}^c(Y)\cong 2\operatorname{Coker}M+\vec{\delta}.
	\]
	Then the application of the formula \eqref{eqn_defn_(qt)_series} to the plumbing leads to the following result:
	\begin{gather}
			\hat{Z}_0(q,t)= \frac{1}{4}q^{\frac{5}{8}}\big[\cdots+\big({-}t^6+t^8-t^{10}+t^{12}-t^{14}+\cdots\big)q^{-3}\nonumber\\
\hphantom{\hat{Z}_0(q,t)=}{}
+\big({-}t^{4}+t^6-t^{8}+t^{10}-t^{12}+\cdots\big)q^{-2}+\big({-}t^2+t^4-t^6+t^{8}-t^{10}+\cdots\big)q^{-1}\nonumber\\
\hphantom{\hat{Z}_0(q,t)=}{} +\big(t^2-2t^4+t^{6}-t^{8}+t^{10}+\cdots\big)q^0+\big(2t^4-3t^6+t^{8}-t^{14}+t^{16}+\cdots\big)q\nonumber\\
\hphantom{\hat{Z}_0(q,t)=}{}+\big({-}t^4+3t^6-2t^8-t^{12}+t^{14}+\cdots\big)q^2+\big(t^4-t^6+2t^8-3t^{10}+t^{12}+\cdots\big)q^3\nonumber\\
\hphantom{\hat{Z}_0(q,t)=}{}+\big(t^{6}-2t^8+3t^{10}-2t^{12}-t^{20}+\cdots\big)q^4+\cdots\big].\label{eqn_result1_qt_series_example_4_1}
		\end{gather}
	
	Now let us compute $ \lim_{t\rightarrow 1} \hat{Z}_0(q,t)$ to recover the $ q $-series. To do this, we need to find out the analytical property of the series $ C_{\xi,n}(t) $, the coefficients of the term $ q^{n} $ in \eqref{eqn_result1_qt_series_example_4_1} from the contribution of the chamber $ \xi\in \mathcal{S} $. By the role of the action $ \text{CT}_{\vec{z}} $ in \eqref{eqn_defn_(qt)_series} it follows that the vectors $ \vec{\ell}=(\ell_1,\ell_2,\ell_3,\ell_4,\ell_5,\ell_6) $ should be of the form
	\[
	\ell_1=\xi_1(2r+1), \qquad\ell_2=\xi_2(2s+1), \qquad\ell_i=\pm1  \quad \text{for } i=3,4,5,6,
	\]
	where $ \xi_1 $ and $ \xi_2 $ are the coordinates of the chamber vector $ \xi=(\xi_1,\xi_2)\in \{\pm 1\}^2 $. Therefore, $ C_{\xi,n}(t) $ is given by
	\begin{equation}\label{eqn_t_series_indefinite_example1}
		C_{\xi,n}(t)= \sum_{r,s=0}^{\infty}\sum_{\varepsilon_i\in\{\pm1\}} \xi_1\xi_2\varepsilon_3\varepsilon_4\varepsilon_5\varepsilon_6 \,t^{2r+2s+2} \Big|_{-\frac{1}{2}-\frac{(\vec{\ell},M^{-1}\vec{\ell})}{4}=\frac{5}{8}+n},
	\end{equation}
	where $ \vec{\ell} $ is the vector of the form
	\[
	\vec{\ell}=\big(\xi_1(2r+1), \xi_2(2s+1), \varepsilon_3,\varepsilon_4,\varepsilon_5,\varepsilon_6\big)
	\]
	and the convention
	\[
 \sum_{r,s}\sum_{\varepsilon_i}\Big|_{Q(\vec{\ell})=0}
	\]
	denotes that the only $ r$, $s $ and $ \varepsilon_i $ satisfying the quadratic equation $ Q(\vec{\ell})=0 $ contribute to the summations.
	
	By swapping the order of sums in \eqref{eqn_t_series_indefinite_example1}, we get
	\begin{equation}\label{eqn_C_(xi,n)_t_indefinite_example1}
		C_{\xi,n}(t)= \sum_{\varepsilon_i\in\{\pm1\}} \xi_1\xi_2\varepsilon_3\varepsilon_4\varepsilon_5\varepsilon_6  \sum_{r,s=0}^{\infty}t^{2r+2s+2} \Big|_{-\frac{1}{2}-\frac{(\vec{\ell},M^{-1}\vec{\ell})}{4}=\frac{5}{8}+n},
	\end{equation}
	which means that given a chamber $ \xi\in \mathcal{S} $ and chosen $ \varepsilon_i $'s, we have the sum of monomials in $ t $ with powers of $ 2r+2s+2 $ from all non-negative integer solutions of the quadratic equation in two variables $ r $ and $ s $
	\[
	-\frac{1}{2}-\frac{\big(\vec{\ell},M^{-1}\vec{\ell}\,\big)}{4}=\frac{5}{8}+n,
	\]
	moreover, the product $\xi_1\xi_2\varepsilon_3\varepsilon_4\varepsilon_5\varepsilon_6$ determines the sign of the sum over $ r $ and $ s $. Therefore, we naturally move on to the solutions of the quadratic Diophantine equations (QDEs). In Appendix~\ref{appendix_QDE}, we briefly review the algorithm to solve the QDEs in two variables for convenience.
	
	In general, the quadratic form from a plumbing with two high-valency vertices
	\[
	\frac{3\sigma-\sum_v m_v}{4}-\frac{\big(\vec{\ell},M^{-1}\vec{\ell}\,\big)}{4}=\Delta_a+n
	\]
	is a quadratic equation with rational coefficients, but multiplying it by $ 4\det M $ we can obtain the one with integer coefficients. For a fixed chamber $ \xi\in \mathcal{S} $ and chosen $ \varepsilon_i $'s, we denote the quadratic equation with integer coefficients by
	\[
	Q_{\xi,\vec{\varepsilon}}=4\det M\left(\frac{3\sigma-\sum_v m_v}{4}-\frac{\big(\vec{\ell},M^{-1}\vec{\ell}\,\big)}{4}-\Delta_a-n\right)=0.
	\]
	
	By putting the solutions of quadratic equations into~\eqref{eqn_C_(xi,n)_t_indefinite_example1}, we obtain the explicit formula for~$ C_{\xi,n}(t) $, for example, $ n=2 $ as follows:
	\begin{gather}
C_{\xi_{++},2}(t) =-t^2+t^4+2t^6+\sum_{k=1}^{\infty}\big(t^{4k^2+6k-4}+t^{4k^2+10k}+t^{4k^2+6k-2}+t^{4k^2+2k-4}\big),\nonumber\\
C_{\xi_{+-},2}(t) =-t^4-3t^6-2t^8+\sum_{k=1}^{\infty}\big({-}t^{4k^2+14k+4}-t^{4k^2+10k-2}-t^{4k^2+6k-4}-t^{4k^2+10k}\big), \!\!\!\label{eqn_C_(xi,2)_t_indefinite_example1}
		\end{gather}
	where $ \xi_{++}=\{+1,+1\} $ and $ \xi_{+-}=\{+1,-1\} $. We notice that $ C_{\xi_{++},n}(t)=C_{\xi_{--},n}(t) $ and $ C_{\xi_{+-},n}(t)=C_{\xi_{-+},n}(t) $ due to the quadratic form. Therefore, it is enough to evaluate the limit of the form
	\[
	\lim_{t\rightarrow 1}\sum_{k=0}^{\infty} t^{a(k+b)^2+c}, \qquad \text{for } a>0.
	\]
	
	First, we change the variable by using $ t={\rm e}^{-\epsilon} $, then we have
	\[
	\lim_{t\rightarrow 1}\sum_{k=0}^{\infty} t^{a(k+b)^2+c}=\lim_{\epsilon\rightarrow 0^{+}} \sum_{k=0}^{\infty} {\rm e}^{-\epsilon(a(k+b)^2+c)}.
	\]
	We recall the Euler--Maclaurin formula
	\begin{gather*}
\sum_{k=N_i}^{N_f}f(k) =\int_{N_i}^{N_f}f(x){\rm d}x+\frac{1}{2}(f(N_f)+f(N_i))\\
\hphantom{\sum_{k=N_i}^{N_f}f(k) =}{}
+\sum_{i=2}^{j}\frac{b_i}{i!}\big[f^{(i-1)}(N_f)-f^{(i-1)}(N_i)\big]-\int_{N_i}^{N_f}\frac{B_j(\{1-x\})}{j!}f^{(j)}(x){\rm d}x,
\end{gather*}
where $ N_i $ and $ N_f $ are real numbers such that $ N_f-N_i $ is a positive integer number, $ B_j $ and~$ b_j$ denote Bernoulli polynomials and numbers, respectively, $ j $ is any positive integer and $ \{x\} $ denotes the fractional part of a real number~$ x $. Applying the Euler--Maclaurin formula, we have%
	\begin{equation*}
		\sum_{k=0}^{\infty} {\rm e}^{-\epsilon(a(k+b)^2+c)}=\int_{0}^{\infty}{\rm e}^{-\epsilon a(x+b)^2-\epsilon c}{\rm d}x+1+O(\epsilon) \qquad \text{as } \epsilon\rightarrow 0,
	\end{equation*}
	then the evaluation of Gaussian integral returns
	\begin{equation}\label{eqn_result_euler_maclaurin_example_4_1}
		\sum_{k=0}^{\infty} {\rm e}^{-\epsilon(a(k+b)^2+c)}=1-b+\frac{1}{2}\sqrt{\frac{\pi}{\epsilon a}}+O\big(\sqrt{\epsilon}\big),
	\end{equation}
	where we have used the error function's Maclaurin series
	\[
	\text{erf}(x)=\frac{2}{\sqrt{\pi}}\int_{0}^{x} {\rm e}^{-y^2}{\rm d}y=\frac{2}{\sqrt{\pi}}\sum_{j=0}^{\infty}\frac{(-1)^j x^{2j+1}}{j!(2j+1)}.
	\]	
	By substituting \eqref{eqn_result_euler_maclaurin_example_4_1} into \eqref{eqn_C_(xi,2)_t_indefinite_example1}, we obtain
	\[
		C_{\xi_{++},2}(\epsilon)=3+\sqrt{\frac{\pi}{\epsilon}}+O\big(\sqrt{\epsilon}\big),\qquad
		C_{\xi_{+-},2}(\epsilon)=-5-\sqrt{\frac{\pi}{\epsilon}}+O\big(\sqrt{\epsilon}\big).
	\]
	Even though the term $ \sqrt{\frac{\pi}{\epsilon}} $ is singular at $ \epsilon\rightarrow 0 $, this singularity will disappear when we take a~sum over all possible chambers. Thus, we see that the coefficient of $ q^{\frac{21}{8}} $ in the recovered $ q $-series is finite and equal to $ -5 $.
	
By using a similar approach, we obtain the following recovered $ q $-series
\[
\hat{Z}_0^R(q)=\lim_{t\rightarrow 1} \hat{Z}_0^R(q,t)=q^{\frac{5}{8}}\big(\cdots+q^{-3}+q^{-2}+q^{-1}-1-q-q^2-q^3-q^4+\cdots\big),
\]
	where the superscript denotes the recovered $q$-series.
\end{Example}

In the previous example, we have seen that for some plumbings without weakly definiteness the $ (q,t) $-series can be recovered into the $ q $-series by computing the limit $ t\rightarrow 1 $. However, this does not hold for all strongly indefinite plumbings. The next example shows a such case.

\begin{Example}
	We consider a reduced plumbing shown in Figure \ref{fig_indefinite_example}. We note that just for simplicity we choose a plumbing whose linking matrix has determinant 1 such that there is a~unique $ \text{Spin}^c $ structure $ a=0 $.
	
	\begin{figure}[t]
		\centering
		\includegraphics[scale=0.4]{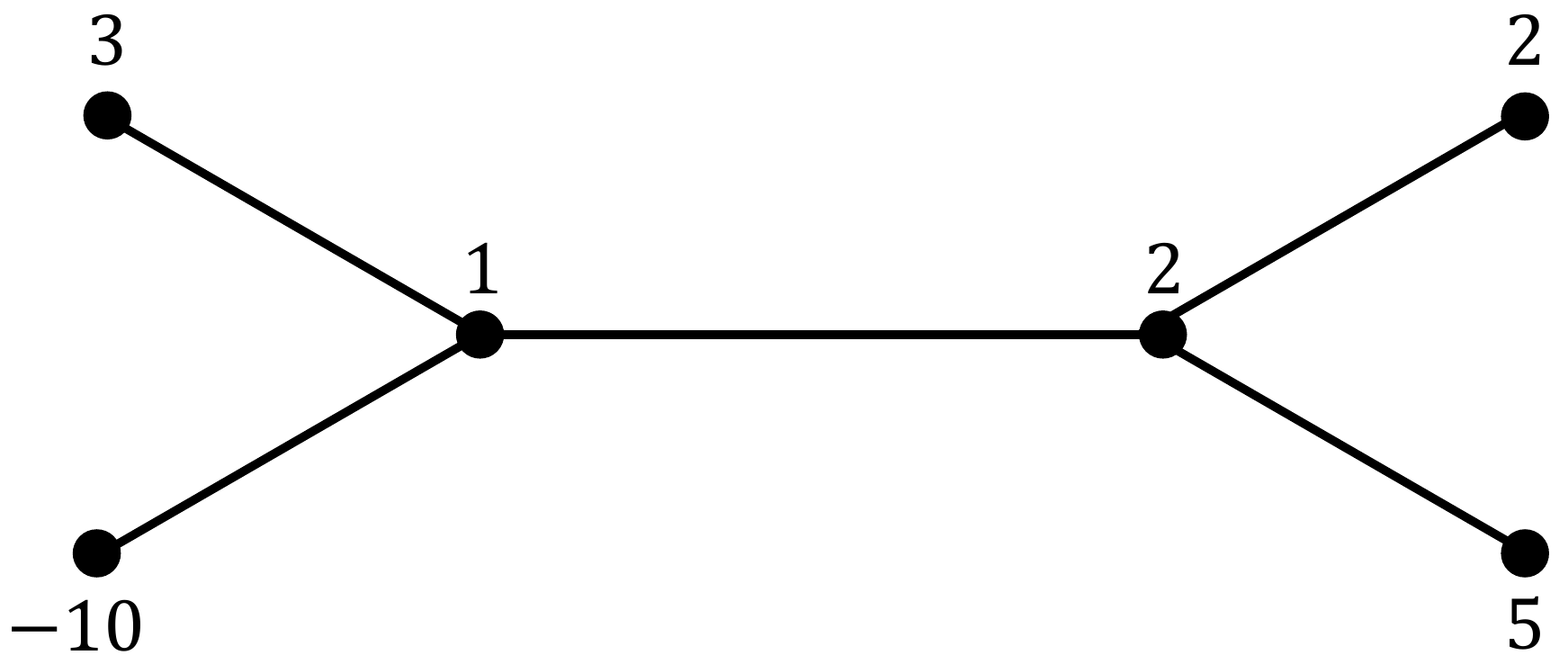}
		\caption{A plumbing graph that has the strongly indefinite property and whose linking matrix has determinant equal to 1.}
		\label{fig_indefinite_example}
	\end{figure}
	For this plumbing, the linking matrix and its inverse are given by
	\[
	M=\begin{pmatrix}
		1&1&1&1&0&0\\
		1&2&0&0&1&1\\
		1&0&3&0&0&0\\
		1&0&0&-10&0&0\\
		0&1&0&0&2&0\\
		0&1&0&0&0&5\\
	\end{pmatrix},\qquad M^{-1}=\begin{pmatrix}
		-390&300&130&-39&-150&-60\\
		300&-230&-100&30&115&46\\
		130&-100&-43&13&50&20\\
		-39&30&13&-4&-15&-6\\
		-150&115&50&-15&-57&-23\\
		-60&46&20&-6&-23&-9
	\end{pmatrix}.
	\]
	From the linking matrix, we obtain the following quantities
	\[
	\pi=\pi(M)=4,\qquad \frac{3\sigma-\sum_{v}m_v}{4}=\frac{3}{4}.
	\]
	Since there is no pseudo bridge in the plumbing, the set $ \mathcal{S} $ of all possible chambers is given by
	\[
	\mathcal{S}=\{(1,1),(1,-1),(-1,1),(-1,-1)\}.
	\]
	Therefore, the formula \eqref{eqn_defn_(qt)_series} reads
	\begin{gather}\label{eqn_defn_IndefiniteExample_(qt)_series}
\hat{Z}_0(q,t) =\frac{1}{4}q^{\frac{3}{4}}\text{CT}_{\vec{z}}\left\{\sum_{r=0}^{\infty}t^{2r+1}\big(z_1^{2r+1}-z_1^{-2r-1}\big)\right. \sum_{s=0}^{\infty}t^{2s+1}\big(z_2^{2s+1}-z_2^{-2s-1}\big)\\
 \left.
 \hphantom{\hat{Z}_0(q,t) =}{}
 \times \left(z_3-\frac{1}{z_3}\right)\left(z_4-\frac{1}{z_4}\right)\left(z_5-\frac{1}{z_5}\right)\left(z_6-\frac{1}{z_6}\right)\sum_{\vec{\ell}\in 2M\mathbb{Z}^6+\vec{\delta}}q^{-\frac{(\vec{\ell},M^{-1}\vec{\ell})}{4}}\prod_{v=1}^{6}z_v^{\ell_v}\right\}.\nonumber
\end{gather}
	The application of the formula \eqref{eqn_defn_IndefiniteExample_(qt)_series} gives the following result:
	\begin{gather*}
			\hat{Z}_0(q,t) =\frac{1}{2}q^{\frac{1}{2}}\big[\cdots+\big(2t^{30}+t^{46}-t^{108}+\cdots\big)q^{-4}+\big(t^8-t^{10}-t^{22}+\cdots\big)q^{-3}\\
\hphantom{\hat{Z}_0(q,t) =}{}
+\big(2t^6-t^8-t^{16}+\cdots\big)q^{-2}+\big(t^8+t^{12}+t^{28}+\cdots\big)q^{-1}\\
\hphantom{\hat{Z}_0(q,t) =}{}+\big(2t^2-t^{12}-t^{564}+\cdots\big)+\big({-}t^2+t^{16}+t^{36}+\cdots\big)q\\
\hphantom{\hat{Z}_0(q,t) =}{}+\big({-}t^6-t^{18}-t^{34}+\cdots\big)q^3+\big({-}t^{12}-t^{20}-t^{60}+\cdots\big)q^4+\cdots\big].
		\end{gather*}
	We note that all the coefficients of $ q^n $ in this result are infinite series in $ t $ because of the strong indefiniteness of the linking matrix.
	
	For a chosen chamber $ \xi=(\xi_1,\xi_2)\in\{\pm1\}^2 $, we know that the vectors $ \vec{\ell}=(\ell_1,\ell_2,\ell_3,\ell_4,\ell_5,\ell_6) $ should be of the form
	\[
	\ell_1=\xi_1(2r+1), \qquad \ell_2=\xi_2(2s+1), \qquad \ell_i=\pm1 \quad \text{for } i=3,4,5,6,
	\]
	and $ C_{\xi,n}(t) $ is given by
	\begin{equation}\label{eqn_C_(xi,n)_t_indefinite_example}
		C_{\xi,n}(t)= \sum_{\varepsilon_i\in\{\pm1\}} \xi_1\xi_2\varepsilon_3\varepsilon_4\varepsilon_5\varepsilon_6 \,\sum_{r,s=0}^{\infty}t^{2r+2s+2} \Big|_{\frac{3}{4}-\frac{(\vec{\ell},M^{-1}\vec{\ell})}{4}=\frac{1}{2}+n}.
	\end{equation}
	
	According to the algorithm in Appendix~\ref{appendix_QDE}, the nonnegative integer solutions of
	\[
	Q_{\xi,\vec{\varepsilon}}=\frac{3}{4}-\frac{\big(\vec{\ell},M^{-1}\vec{\ell}\,\big)}{4}-\frac{1}{2}-n=0
	\]
	are determined by a finite number of families of solutions, and each family $ \{(r_k,s_k)\}_{k} $ of solutions has the following form:
	\begin{gather*}
			r_k =\alpha_r A^k +\beta_r A^{-k}+\eta_r,\qquad
			s_k =\alpha_s A^k +\beta_s A^{-k}+\eta_s,
		\end{gather*}
	where $ k\geq k_0\in \mathbb{Z}_{+} $. Here the real numbers $ \alpha_r$, $\alpha_s$, $\beta_r$, $\beta_s$, $\eta_r$, $\eta_s $ are constants which do depend on the choice of the chamber $ \xi $ and the vector $ \vec{\varepsilon} $. However, the $ A $ does not depend on~$ \xi $ and~$ \vec{\varepsilon} $. This is because all QDEs $ Q_{\xi,\vec{\varepsilon}}=0 $ from different $ \xi $ and $ \vec{\varepsilon} $ can be transformed into the same generalized Pell's equation. From the requirement for non-negative solutions, it follows that $ \alpha_r $ and $ \alpha_s $ are positive real numbers and $ A>1 $.
	
	For each family of solutions, we obtain a series in $ t $ by \eqref{eqn_C_(xi,n)_t_indefinite_example} as following:
	\[
	\sum_{k\geq k_0} t^{2(\alpha_r+\alpha_s)A^k+2(\beta_r+\beta_s)A^{-k}+2(\eta_r+\eta_s)}.
	\]
	In order to compute the limit $ t\rightarrow 1 $, we use the substitution $ t\rightarrow {\rm e}^{-\epsilon} $
	\[
	\sum_{k\geq k_0} {\rm e}^{-2\epsilon[(\alpha_r+\alpha_s)A^k+(\beta_r+\beta_s)A^{-k}+\eta_r+\eta_s]}.
	\]
	
	By applying doubly exponential series transformations of Ramanujan \cite{ramanujan}, we can estimate the above series as follows:
	\begin{gather}
\sum_{k\geq k_0} {\rm e}^{-2\epsilon[(\alpha_r+\alpha_s)A^k+(\beta_r+\beta_s)A^{-k}+\eta_r+\eta_s]}\nonumber\\
\quad{}\sim\frac{1}{2}+\frac{1}{\log A}\left[-\log\epsilon-\gamma-\log2-\log(\alpha_r+\alpha_s)\right]-k_0+O(\epsilon)\label{eqn_estimation_epsilon_exponential_series}\\
\quad{}+2\sum_{k=1}^{\infty}\sqrt{\frac{\log A}{2k \sinh{\frac{2k\pi^2}{\log A}}}}\cos\left(\frac{2k\pi}{\log A}\log\frac{k\pi}{\epsilon(\alpha_r+\alpha_s)\log A}-\frac{2k\pi}{\log A}-\frac{\pi}{4}-\frac{B_2\log A}{4k\pi}+\cdots\right),\nonumber
		\end{gather}
	where $ \gamma $ is the Euler's constant $ \gamma=0.57721\dots $ and $ B_k $ denotes the $ k $th Bernoulli number defined by
	\[
	\frac{z}{{\rm e}^z-1}=\sum_{k=0}^{\infty}\frac{B_k}{k!}z^k, \qquad |z|<2\pi.	
	\]
	
	Now let us analyse the terms in the right side of \eqref{eqn_estimation_epsilon_exponential_series}. First, the second term proportional to $ \log \epsilon $ has a singularity at $ \epsilon\rightarrow 0 $, but this term will be cancelled out by taking a sum of $ C_{\xi,n}(t) $ over all choices of chambers $ \xi\in \mathcal{S} $, which allows us to ignore this term. Next, the infinite series of oscillating terms has a singularity as well. Furthermore, this terms do not disappear by a~sum of $ C_{\xi,n}(t) $. Thus, we conclude that due to such series of oscillating terms the limit $ t\rightarrow 1 $ is ill-defined and the recovered $ q $-series does not exist.
	
\end{Example}

\begin{Example}
	\label{example_nonreduced_plumbing}
	In this example, we consider a non-reduced plumbing shown in Figure \ref{fig_nonreduced235_example}, which can be obtained by Neumann moves from the one in Figure \ref{fig_Sigma_235}. Hence, the 3-manifold realized by the plumbing graph in Figure \ref{fig_nonreduced235_example} is homeomorphic to the Brieskorn sphere $ \overline{\Sigma(2,3,5)} $. One can easily check that the high-valency vertex with weight 3, placed in the right side of the plumbing, is a pseudo high-valency vertex because the branches are both pseudo.
	
	\begin{figure}[t]
		\centering
		\includegraphics[scale=0.4]{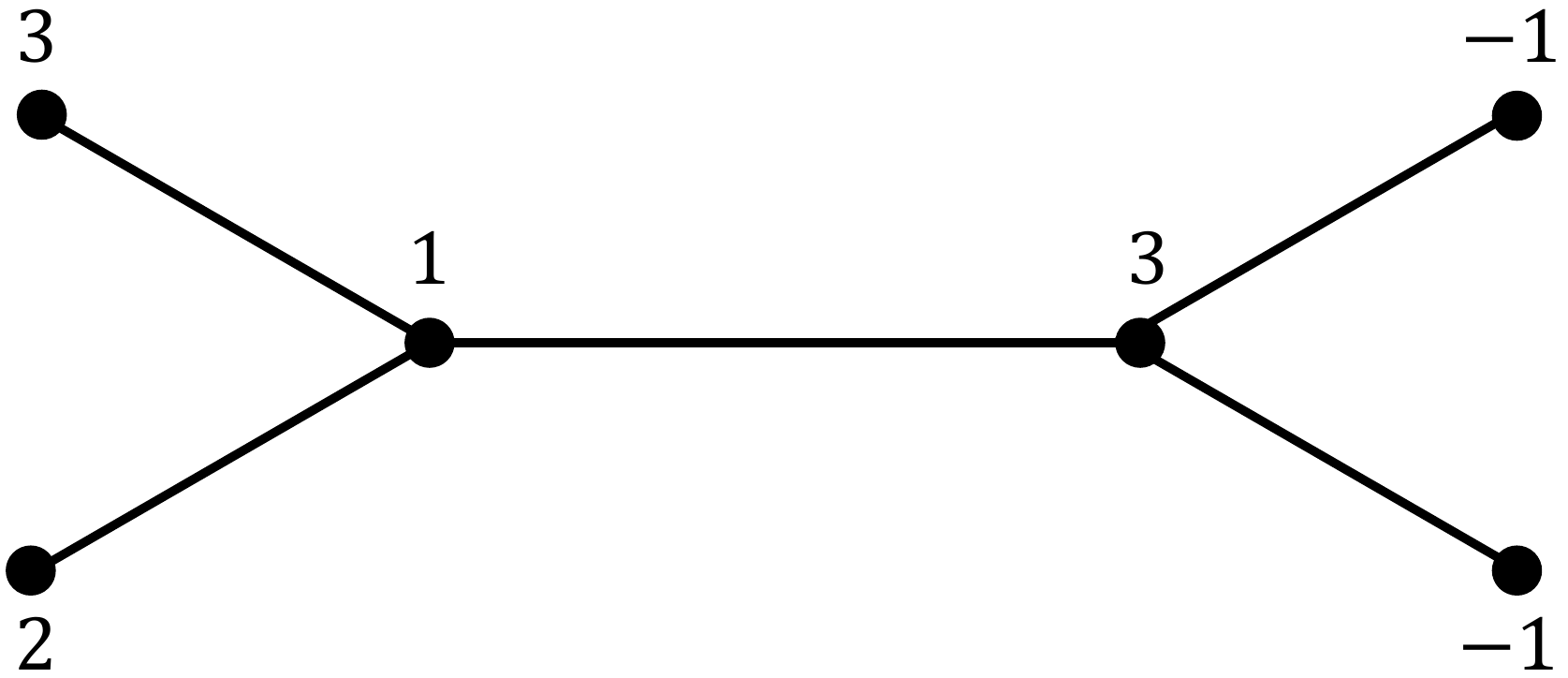}
		\caption{A non-reduced plumbing graph that realizes the 3-manifold $ \overline{\Sigma(2,3,5)}$.}\label{fig_nonreduced235_example}
	\end{figure}
	
	We have already computed the $ (q,t) $-series \eqref{eqn_(q,t)_result_Sigma235} in the last part of Section \ref{section_(qt)_series}. Now we are going to compute the $ (q,t) $-series for the non-reduced plumbing in Figure \ref{fig_nonreduced235_example} and compare it with the result in \eqref{eqn_(q,t)_result_Sigma235}.
	
	The linking matrix and its inverse of the plumbing are given by
	\[
	M=\begin{pmatrix}
		1&1&1&1&0&0\\
		1&3&0&0&1&1\\
		1&0&3&0&0&0\\
		1&0&0&2&0&0\\
		0&1&0&0&-1&0\\
		0&1&0&0&0&-1
	\end{pmatrix}, \qquad M^{-1}=\begin{pmatrix}
		-30&6&10&15&6&6\\
		6&-1&-2&-3&-1&-1\\
		10&-2&-3&-5&-2&-2\\
		15&-3&-5&-7&-3&-3\\
		6&-1&-2&-3&-2&-1\\
		6&-1&-2&-3&-1&-2
	\end{pmatrix}.
	\]
	
	The formula \eqref{eqn_defn_(qt)_series} reads
	\begin{gather*}
\hat{Z}_0(q,t)=\frac{1}{4}q^{-\frac{7}{4}}\text{CT}_{\vec{z}}\left\{\sum_{r=0}^{\infty}t^{2r+1}\big(z_1^{2r+1}-z_1^{-2r-1}\big)\right.
\left(z_3-\frac{1}{z_3}\right)\left(z_4-\frac{1}{z_4}\right)\\
\hphantom{\hat{Z}_0(q,t)=}{}
\times \sum_{s=0}^{\infty}t^{2s+1}\left[z_2^{2s+1}\left(z_5 t-\frac{1}{z_5 t}\right)\left(z_6 t-\frac{1}{z_6 t}\right)-z_2^{-2s-1}\left(\frac{z_5}{t}-\frac{t}{z_5}\right)\left(\frac{z_6}{t}-\frac{t}{z_6}\right)\right]\\
\left.
\hphantom{\hat{Z}_0(q,t)=}{} \times\sum_{\vec{\ell}\in 2M\mathbb{Z}^6+\vec{\delta}}q^{-\frac{(\vec{\ell},M^{-1}\vec{\ell})}{4}}\prod_{v=1}^{6}z_v^{\ell_v}\right\},
	\end{gather*}
	and its computation gives an interesting result as follows:
	\begin{gather*}
		\hat{Z}_0(q,t)=\frac{1+t^2}{2}q^{-\frac{3}{2}}\big(1 - {q} - q^{3} - q^{7} + t q^{8} + q^{14} + q^{20} + t^2 q^{29} - q^{31} - t^2 q^{42} +\cdots\big).
	\end{gather*}
	Here one might be curious about the result that coefficients of monomials in $ q $ are finite polynomials in $ t $ even though the plumbing is strongly indefinite. The key to the answer is that there are cancellations between each pair of vectors $ \vec{\ell}_1=(\dots, +1) $ and $ \vec{\ell}_2=(\dots, -1) $, where~$ \pm 1 $ denotes the element of the vector $ \vec{\ell} $ associated to valency one vertex on any pseudo branch.
	
	We notice that except an extra factor $ \frac{1+t^2}{2t} $, this result completely agrees with \eqref{eqn_(q,t)_result_Sigma235}, furthermore the recovered $ q $-series by taking $ t=1 $ are the same.
\end{Example}

Let us reiterate that Theorem \ref{theorem_invariance} says that the $ (q,t) $-series for a reduced plumbing defined in~\eqref{eqn_defn_(qt)_series} is unchanged by the Neumann moves, thus the recovered $ q $-series also remains unchanged if the limit $ t\rightarrow 1 $ exists. The Example~\ref{example_nonreduced_plumbing} tells us that the $ (q,t) $-series for a non-reduced plumbing is changed by Neumann moves, but the recovered $ q $-series might be still preserved. This property indeed holds for non-reduced plumbings with at most two high-valency vertices.

\begin{Proposition}
	\label{prop_recovered_q_series_preserved}
	Assume that the $ (q,t) $-series for a given plumbing can be finitely recovered $ q $-series in a sense that the limit $ t\rightarrow 1 $ is finite. Then the recovered $q$-series is preserved by arbitrary Neumann moves even for non-reduced plumbings with at most two high-valency vertices.
\end{Proposition}
\begin{proof}
	Since the $(q,t)$-series gets changed only when a Neumann move creates a pseudo high-valency vertex, it is sufficient to show that the recovered $q$-series remains unchanged for such a Neumann move. Therefore, we are going to consider a case shown in Figure \ref{fig_nonreduced_two_hvv_proposition}, where the left plumbing (a) is a reduced one with one high-valency vertex weighted by $ m_1 $ and the right plumbing (b) has an extra pseudo high-valency vertex created by a Neumann move of type (b) in Figure \ref{fig_neumann_moves}. The cases for other types of Neumann moves creating a new pseudo high-valency vertex can be handled in a similar way.
	
	\begin{figure}[t]
		\centering
		\includegraphics[scale=0.4]{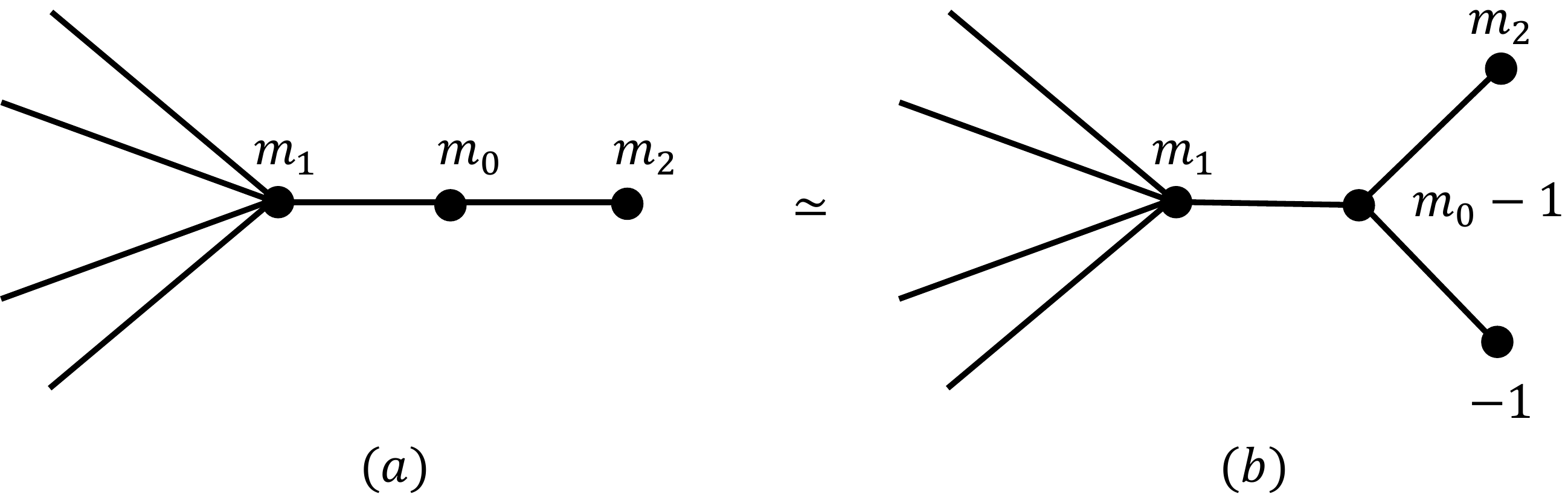}
		\caption{A reduced plumbing (a) with one high-valency vertex turns out to be a non-reduced plumbing~(b) with two high-valency vertices by a Neumann move.}
		\label{fig_nonreduced_two_hvv_proposition}
	\end{figure}
	
	Let $ \Gamma $ be the left plumbing and $ \Gamma' $ be the right one in Figure \ref{fig_nonreduced_two_hvv_proposition}. We use a prime to denote the quantities for the right plumbing. Let us write a vector $ \vec{\ell} $ by
	\[
	\vec{\ell}=\big(\vec{\ell}_1,0,\ell_2\big),
	\]
	where $ \vec{\ell}_1 $ corresponds to the left part of the graph $ \Gamma $ including the vertex $ m_1 $. As we have already seen in the proof of Theorem \ref{theorem_invariance}, we have corresponding vectors for $ \Gamma' $
	\[
	\vec{\ell}'_{\pm}=\big(\vec{\ell}_1, \pm1, \ell_2, \mp 1\big),
	\]
	which give the exactly same contribution as $ \vec{\ell} $. However, due to the existence of pseudo high-valency vertex in $ \Gamma' $, we have the following extra vectors $ \vec{\ell}' $ corresponding to $ \vec{\ell} $
	\[
	\vec{\ell}'_{1,\pm}=\big(\vec{\ell}_1, 2s+1,\ell_2,\pm1\big), \qquad \vec{\ell}'_{2,\pm}=\big(\vec{\ell}_1, -2s-1,\ell_2,\pm1\big)
	\]
	for $ s\in \mathbb{N} $. It is immediate that those infinitely many extra vectors generate extra $ t $-series. Therefore, we need to show that the limit $ t \rightarrow 1 $ of those extra $ t $-series is equal to zero.
	
	Let us consider the extra $ t $-series coming from the vectors $ \vec{\ell}'_1 $. The case for the vectors $ \vec{\ell}'_2 $ is similar. For some $ n\in \mathbb{Z} $ and $ \Delta_a\in \mathbb{Q} $, the $ t $-series coefficient of $ q^{\Delta_a+n} $ from those vectors is determined by
	\[
	\sum_{r=0}^{\infty}\sum_{s=1}^{\infty}\left\{\left.t^{2r+2s+p_r+2}\right|_{Q(\vec{\ell}'_{1,+})}-\left.t^{2r+2s+p_r}\right|_{Q(\vec{\ell}'_{1,-})}\right\},
	\]
	where $ p_r=\deg(v_1)-2 $ and $ Q\big(\vec{\ell}'\big) $ denotes the following quadratic equation
	\[
	\frac{3\sigma'-\sum_I m'_I}{4}-\frac{\big(\vec{\ell}',{M'}^{-1}\vec{\ell}'\big)}{4}=\Delta_a+n.
	\]
	This means that those extra $ t $-series would be zero after taking the limit $ t\rightarrow 1 $ if the non-negative solutions of two quadratic equations $ Q\big(\vec{\ell}'_{1,\pm}\big) $ are expressed by the families of solutions with the same patterns.
	
	Indeed, this is the case. As we have mentioned in Appendix~\ref{appendix_QDE}, the first step to solve general quadratic Diophantine equation \eqref{general_diophantine_equation} is to transform it into the form of generalized Pell's equation~\eqref{generalized_pell_equation} by using the following transformations
	\[
	D=b^2-4ac,\qquad E=bd-2ae,\qquad F=d^2-4af,\qquad N=E^2-DF.
	\]
	Therefore, the solutions of \eqref{general_diophantine_equation} are expressed by the factors $ D$, $E $ and $ F $ together with the solutions of the generalized Pell's equation. We have already seen that two quadratic equations~$ Q\big(\vec{\ell}'_{1,\pm}\big) $ are transformed into the same generalized Pell's equation. Moreover, it is an algebraic exercise that these equations have the same transforming factors~$ D$,~$E $ and~$ F $. Thus, the solutions of two quadratic equations have the same patterns, which completes the proof.
\end{proof}

Based on Proposition \ref{prop_recovered_q_series_preserved}, we naturally suggest the following conjecture:

\begin{Conjecture}\label{conjecture}
The recovered $ q $-series, obtained from the $ (q,t) $-series in the sense of $ t\rightarrow 1 $, is an invariant of $3$-manifolds realized by arbitrary $($i.e., possibly non-reduced$)$ plumbings.
\end{Conjecture}

\begin{Remark}As we have seen in Sections~\ref{section_(qt)_series} and \ref{section_recovering_q_series}, $ \hat{Z}_a(q, t) $ is defined for reduced plumbings and it has a nice property that it is not only an invariant of 3-manifolds realized by reduced plumbings but also a refined version of $ \hat{Z}_a(q) $, i.e., $ \lim_{t\rightarrow 1} \hat{Z}_a(q,t) =\hat{Z}_a(q)$ for weakly negative definite plumbings. Furthermore, it provides a key to define $ \hat{Z}_a(q) $ even for strongly indefinite plumbings if its limit as $ t\rightarrow 1 $ exists.
	
	For non-reduced plumbings, it is dependent on the presentation of plumbings and it undergoes changed by the Neumann moves. However, as we demonstrates in Example~\ref{example_nonreduced_plumbing} and Proposition~\ref{prop_recovered_q_series_preserved}, the limit $ \lim_{t\rightarrow 1} \hat{Z}_a(q,t) $, if exists, is still an invariant for non-reduced plumbings with at most two high-valency vertices. Thus, if Conjecture~\ref{conjecture} holds, then we could construct an invariant of all tree plumbed 3-manifolds by obtaining the recovered $ q $-series under the assumption of the existence of the limit~$ t\rightarrow 1 $.
\end{Remark}

\section{Connected sum of plumbed 3-manifolds}
\label{section_connected_sum}

Let $ \Gamma_1 $ and $ \Gamma_2 $ be plumbing graphs and $ Y_1$, $Y_2 $ be 3-manifolds realized by $ \Gamma_1$, $\Gamma_2 $, respectively. It is well-known that the connected sum $ Y_1\sharp Y_2 $ can be realized by the disjoint union $ \Gamma_1 \sqcup \Gamma_2 $. The following proposition \cite{neumann} says that the disjoint union of finite number of plumbing graphs is equivalent a certain single plumbing graph.

\begin{Proposition}
	\label{prop_disjoint_union_neumann}
	If a plumbing graph $ \Gamma $ has the form of Figure~{\rm \ref{fig_disjoin_union}}, then the $3$-manifold $ Y(\Gamma) $ is homeomorphic to the connected sum of $ Y(\Gamma_1), Y(\Gamma_2), \dots, Y(\Gamma_n) $.
\begin{figure}[t]		\centering
\includegraphics[scale=0.4]{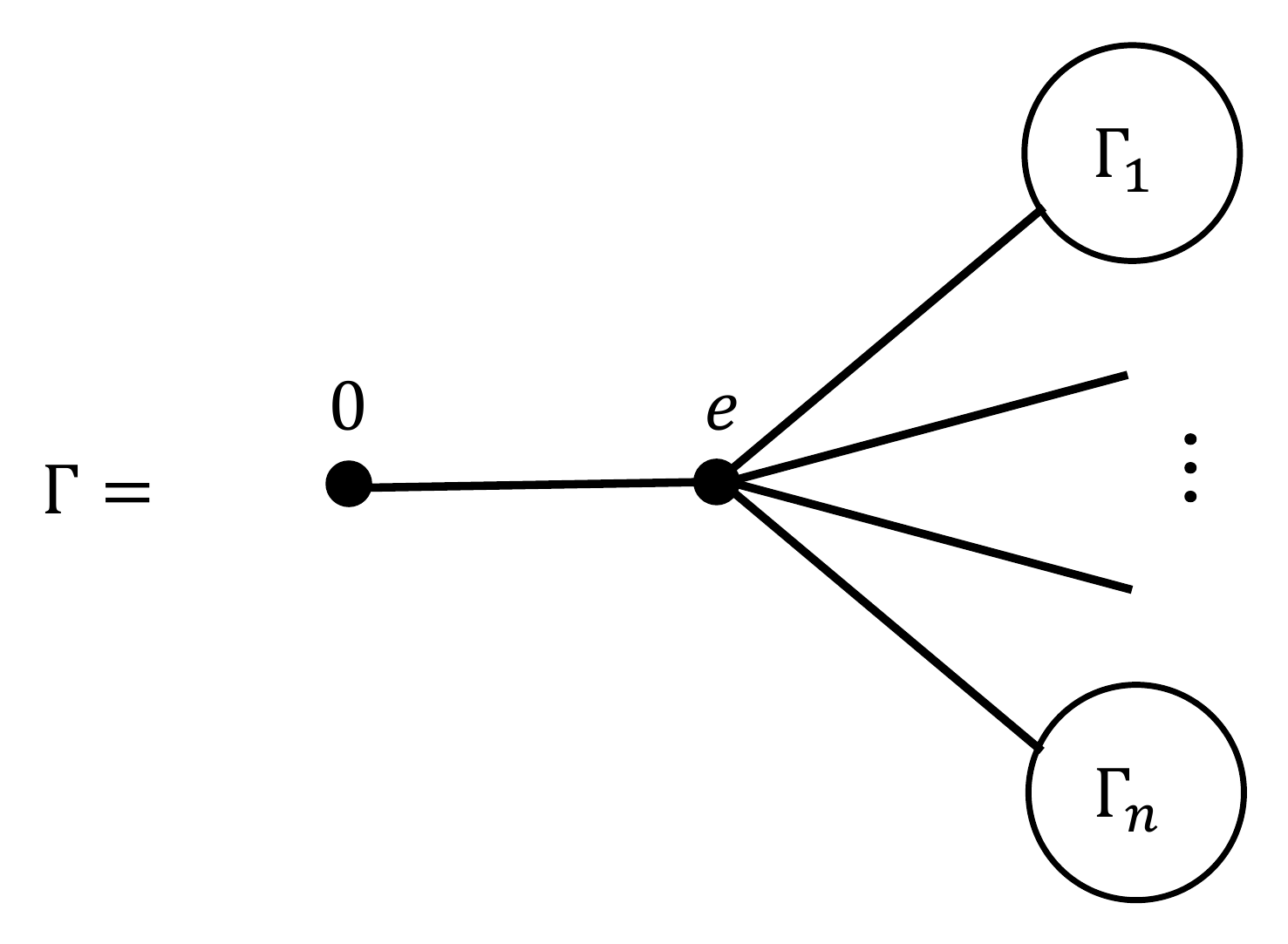}
\caption{A plumbing graph $ \Gamma $ that is equivalent to the disjoint union of $ \Gamma_1,\Gamma_2,\dots, \Gamma_n $. Here the weight~$ e $ can be any integer.}		\label{fig_disjoin_union}
	\end{figure}
\end{Proposition}

Let $ \hat{Z}_{a_j}(\Gamma_j; q,t) $ be the $ (q,t) $-series for $ \Gamma_j $ with $ \text{Spin}^c $ structure $ a_j $, where $ j=1,2. $ Then, according to Proposition~\ref{prop_disjoint_union_neumann}, the connected sum of $ Y_1=Y(\Gamma_1) $ and $ Y_2=Y(\Gamma_2) $ can be realized by a plumbing $ \Gamma $ in the form of Figure~\ref{fig_disjoin_union}. Just for simplicity, we assume that $ \Gamma_j $ is connected to the vertex $ e $ through a valency one vertex of a non-pseudo branch in~$ \Gamma_j $. This assumption is just to avoid creating a new high-valency vertex or a new pseudo bridge.

We now present a formula for the $ (q,t) $-series $ \hat{Z}_a(\Gamma; q,t) $.

\begin{Proposition}
	The $ (q,t) $-series of $ \Gamma $ can be expressed in terms of the product of $ \hat{Z}_{a_j}(\Gamma_j; q,t) $, $ j=1,2 $ as following:
	\begin{equation}\label{eqn_disjoint_union_formula}
		\hat{Z}_a(\Gamma; q,t)=\frac{1}{2}\sum_{r=0}^{\infty} t^{2r+1}\big(q^{-\frac{2r+1}{2}}-q^{\frac{2r+1}{2}}\big) \cdot \hat{Z}_{a_1}(\Gamma_1; q,t)\hat{Z}_{a_2}(\Gamma_2; q,t).
	\end{equation}
	It is worth to notice here that the series
	\[
	\frac{1}{2}\sum_{r=0}^{\infty} \big(q^{-\frac{2r+1}{2}}-q^{\frac{2r+1}{2}}\big)
	\]
	is the symmetric expansion of a rational function
	\[
	\frac{1}{q^{1/2}-q^{-1/2}},
	\]
	that is just the inverse of the $ q $-series of $3$-sphere.
\end{Proposition}
\begin{proof}
	Suppose that the linking matrices of $ \Gamma_1 $ and $ \Gamma_2 $ are given by
	\[
	M_1=\begin{pmatrix}
		a_{11}&\cdots &a_{1n}\\
		\vdots&\ddots&\vdots\\
		a_{1n}&\cdots&a_{nn}
	\end{pmatrix},\qquad M_2=\begin{pmatrix}
		b_{11}&\cdots &b_{1m}\\
		\vdots&\ddots&\vdots\\
		b_{1m}&\cdots&b_{mm}
	\end{pmatrix},
	\]
	then the linking matrix of $ \Gamma $ is expressed by
	\begin{equation}\label{eqn_linkingmatrix_disjoint_union}
		M=\begin{pmatrix}
			a_{11}&\cdots &a_{1n}&0&\cdots&\cdots&\cdots&0\\
			\vdots&\ddots&\vdots&\vdots&\ddots&\ddots&\ddots&\vdots\\
			a_{1n}&\cdots&a_{nn}&0&1&0&\cdots&0\\
			0&\cdots&0&0&1&0&\cdots&0\\
			0&\cdots&1&1&e&1&0&\cdots\\
			0&\cdots&0&0&1&b_{11}&\cdots&b_{1m}\\
			\vdots&\ddots&\ddots&\ddots&\vdots&\vdots&\ddots&\vdots\\
			0&\cdots&\cdots&\cdots&0&b_{1m}&\cdots&b_{mm}
		\end{pmatrix}.
	\end{equation}
	Also, \eqref{eqn_linkingmatrix_disjoint_union} implies
	\begin{gather}
			\det M =-\det M_1\cdot \det M_2,\nonumber\\
			\sigma(M) =\sigma(M_1)+\sigma(M_2),\nonumber\\
			\pi(M) =\pi(M_1)+\pi(M_2)+1.\label{eqn_properties_linkingmatrix_disjoint_union}
		\end{gather}
	Furthermore, the inverse matrix $ M^{-1} $ reads as following:
	\begin{equation}\label{eqn_inverse_linkingmatrix_disjoint_union}
		M^{-1}=\begin{pmatrix}
			A_{11}&\cdots&A_{1n}&-A_{1n}&0&&&\\
			\vdots&\ddots&\vdots&\vdots&\vdots&&&\\
			A_{1n}&\cdots&A_{nn}&-A_{nn}&0&&&\\
			-A_{1n}&\cdots&-A_{nn}&A_{nn}+B_{11}-e&1&-B_{11}&\cdots&-B_{1m}\\
			0&\cdots&0&1&0&0&\cdots&0\\
			&&&-B_{11}&0&B_{11}&\cdots&B_{1m}\\
			&&&\vdots&\vdots&\vdots&\ddots&\vdots\\
			&&&-B_{1m}&0&B_{1m}&\cdots&B_{mm}
		\end{pmatrix},
	\end{equation}
	where
	\[
	M^{-1}_1=\begin{pmatrix}
		A_{11}&\cdots &A_{1n}\\
		\vdots&\ddots&\vdots\\
		A_{1n}&\cdots&A_{nn}
	\end{pmatrix},\qquad M^{-1}_2=\begin{pmatrix}
		B_{11}&\cdots &B_{1m}\\
		\vdots&\ddots&\vdots\\
		B_{1m}&\cdots&B_{mm}
	\end{pmatrix}
	\]
	are the inverse matrices of $ M_1 $ and $ M_2 $, respectively. From \eqref{eqn_properties_linkingmatrix_disjoint_union}, we have following relations:
	\begin{gather*}
			(-1)^{\pi(M)}=-(-1)^{\pi(M_1)}\cdot(-1)^{\pi(M_2)},\\
			q^{\frac{3\sigma-\sum_I m_I}{4}}=q^{\frac{3\sigma_1-\sum_{I_1} m_{I_1}}{4}}\cdot q^{\frac{3\sigma_2-\sum_{I_2} m_{I_2}}{4}}\cdot q^{-\frac{e}{4}}.
		\end{gather*}
	
	Now let us write vectors $ \vec{\ell} $ as
	\[
	\vec{\ell}=\big(\ell_1,\dots,\ell_n\pm1,\pm1,\ell_e,\ell_1'\pm1,\dots,\ell_m'\big),
	\]
	with the entry $ \ell_e $ being for the vertex $ e $ and $ \pm1 $ being for the vertex weighted by $ 0 $. Then, it follows from the explicit expression~\eqref{eqn_inverse_linkingmatrix_disjoint_union} for $ M^{-1} $ that we have
	\begin{equation*}
		\big(\vec{\ell},M^{-1}\vec{\ell}\,\big)=\big(\vec{\ell_1},M^{-1}_1\vec{\ell_1}\big)+\big(\vec{\ell_2'},M^{-1}_2\vec{\ell_2'}\big)-e-2\ell_e,
	\end{equation*}
	which implies
	\begin{equation*}
		\Theta_a^{-M}(\vec{z})=\Theta_{a_1}^{-M_1}(\vec{z_1})\left(z_{n}-\frac{1}{z_n}\right)\Theta_{a_2}^{-M_2}(\vec{z_2})\left(z_{1}'-\frac{1}{z_1'}\right)\sum_{\ell_e}q^{\frac{e+2\ell_e}{4}}z_{e}^{\ell_e}.
	\end{equation*}
	
	Putting all together into the defining formula \eqref{eqn_defn_(qt)_series}, it is straightforward to get \eqref{eqn_disjoint_union_formula}.
\end{proof}

\appendix
\section{Quadratic Diophantine equation in two variables}
\label{appendix_QDE}
We are going to review here how to solve the general quadratic Diophantine equation
\begin{equation}\label{general_diophantine_equation}
	ax^2+bxy+cy^2+dx+ey+f=0,
\end{equation}
where the coefficients are all integers, i.e., $ a,b,c,d,e,f\in \mathbb{Z} $. We assume that not all of $ a$, $b$, $c $ are zero. The classical method of solving \eqref{general_diophantine_equation} was firstly given by Lagrange over 300 years ago.

Let $ D=b^2-4ac$, $E=bd-2ae $ and $ F=d^2-4af $. We will examine \eqref{general_diophantine_equation} in several cases, mainly depending on the values of $ D $.

At first, suppose $ D=0 $, then without loss of generality, we can assume $ a\neq 0 $. Equation~\eqref{general_diophantine_equation} can be written as
\begin{equation*}
	Y^2=2Ey+F,
\end{equation*}
where $ Y=2ax+by+d $. In the case of $ E=0 $ and $ F=0 $, then there could be either a line of integer solutions for \eqref{general_diophantine_equation} if $ \text{gcd}(2a,b) $ divides $ d $, or no solution otherwise. It is clear that there is no integer solution of \eqref{general_diophantine_equation} if $ E=0 $ and $ F<0 $ or if $ E=0 $ and $ F $ is positive but not a perfect integral square. If $ E=0 $ and $ F $ is a perfect square, then one can obtain either a line of integer solutions if $ \text{gcd}(2a,b) $ divides $ \sqrt{F}\pm d $ , or no solution otherwise. In the case of $ E\neq 0 $, \eqref{general_diophantine_equation} boils down to the congruence
\begin{equation}\label{diophantine_equation_congruence}
	Y^2 \equiv F \qquad \bmod  |2E|.
\end{equation}
Clearly, there are no integer solutions for \eqref{general_diophantine_equation} if \eqref{diophantine_equation_congruence} does not have integer solutions. In the case that \eqref{diophantine_equation_congruence} has integer solutions, it is not difficult to observe that the integer solutions $ (x,y) $ of \eqref{general_diophantine_equation} lie on a parabola if they exist.

Now suppose that $ D\neq0 $. Then by putting $ N=E^2-DF $, \eqref{general_diophantine_equation} turns out to be
\begin{equation}\label{generalized_pell_equation}
	X^2-DY^2=N,
\end{equation}
where $ X=Dy+E $. It is clear that \eqref{generalized_pell_equation} has a finite number of solutions if $ D<0 $, or if $ D $ is a positive perfect integral square, or if $ N=0 $. Therefore, the only remaining non-trivial case happens when $ N\neq0 $ and $ D>0 $ is not a perfect square. In this case, \eqref{generalized_pell_equation} is called a \textit{generalized Pell's equation} and it has infinite number of solutions if it has at least one.

\subsection{Pell's equation}
A \textit{Pell's equation} is a Diophantine equation of the form
\begin{equation}\label{pell_equation}
	T^2-DU^2=1,
\end{equation}
where $ D $ is a positive integer that is not a perfect square. The trivial solutions are $ (T,U)=(\pm1,0) $. Among all non-trivial solutions, the pair $ (t,u) $ is called the \textit{fundamental solution} if~$ t $ and~$ u $ are both positive and they are minimal. It is well-known from the work of Lagrange in 1768 that any Pell's equation has a non-trivial solution. Furthermore, once we find a solution, there is a way to generate all the solutions of \eqref{pell_equation}. A pair $ (T,U) $ of positive integers is a solution of~\eqref{pell_equation} if and only if there exists $ n\in \mathbb{N} $ such that
\begin{equation}\label{generating_solution_pell}
	T+U\sqrt{D}=\big(t+u\sqrt{D}\big)^n.
\end{equation}
Sometimes it might be much more convenient to express \eqref{generating_solution_pell} in a matrix form as follows:
\begin{align*}
		\begin{pmatrix}
			T_n\\U_n
		\end{pmatrix}&=\begin{pmatrix}
			t&Du\\
			u&t
		\end{pmatrix}^n\begin{pmatrix}
			1\\0
		\end{pmatrix}\\
		&=\begin{pmatrix}
			\sqrt{D}&-\sqrt{D}\\
			1&1
		\end{pmatrix}\begin{pmatrix}
			\big(t+u\sqrt{D}\big)^n&0\\
			0&\big(t-u\sqrt{D}\big)^n
		\end{pmatrix}\begin{pmatrix}
			\sqrt{D}&-\sqrt{D}\\
			1&1
		\end{pmatrix}^{-1}\begin{pmatrix}
			1\\0
		\end{pmatrix}\\
		&=\begin{pmatrix}
			\frac{1}{2}\big(\big(t+u\sqrt{D}\big)^n+\big(t-u\sqrt{D}\big)^n\big)\vspace{1mm}\\
			\frac{1}{2\sqrt{D}}\big(\big(t+u\sqrt{D}\big)^n-\big(t-u\sqrt{D}\big)^n\big)
		\end{pmatrix},
	\end{align*}
where in the second line we have diagonalized the matrix
\[
\begin{pmatrix}
	t&Du\\
	u&t
\end{pmatrix}=\begin{pmatrix}
	\sqrt{D}&-\sqrt{D}\\
	1&1
\end{pmatrix}\begin{pmatrix}
	t+u\sqrt{D}&0\\
	0&t-u\sqrt{D}
\end{pmatrix}\begin{pmatrix}
	\sqrt{D}&-\sqrt{D}\\
	1&1
\end{pmatrix}^{-1}
\]
by using its eigenvectors and eigenvalues. Thus, any positive solution of \eqref{pell_equation} has the form of
\begin{gather}
		T_n =\frac{1}{2}\big(\big(t+u\sqrt{D}\big)^n+\big(t-u\sqrt{D}\big)^n\big),\nonumber\\
		U_n =\frac{1}{2\sqrt{D}}\big(\big(t+u\sqrt{D}\big)^n-\big(t-u\sqrt{D}\big)^n\big),\label{solution_pell_final_form}
	\end{gather}
and clearly $ (\pm T_n, \pm U_n) $ will also be solutions.

\subsection{Generalized Pell's equation}
So far we have seen that Pell's equation \eqref{pell_equation} has an infinite number of solutions. Using those solutions one can describe a method to obtain all solutions of a generalized Pell's equation \eqref{generalized_pell_equation}. Indeed, Lagrange showed that every solution of \eqref{generalized_pell_equation} can be expressed by a power of $ t+u\sqrt{D} $ times $ X+Y\sqrt{D} $, that is,
\begin{equation}\label{pell_multiple}
	X_n+Y_n\sqrt{D}=\big(X+Y\sqrt{D}\big)\big(t+u\sqrt{D}\big)^n\qquad \text{for some } n\in \mathbb{Z},
\end{equation}
where $ (t,u) $ are the fundamental solution of the corresponding Pell's equation $ T^2-DU^2=1 $ and $ (X,Y) $ is a solution of \eqref{generalized_pell_equation} such that
\[
	|X|\leq \sqrt{|N|} \sqrt{t+u\sqrt{D}} /2 \qquad\text{and}\qquad |Y|\leq \sqrt{|N|} \sqrt{t+u\sqrt{D}} /\big(2\sqrt{D}\big).
\]
In other words, there exists a finite set $ \mathcal{S} $ of solutions for \eqref{generalized_pell_equation} such that every solution $ (X_n,Y_n) $ can be obtained by \eqref{pell_multiple} with some $ (X,Y)\in \mathcal{S} $. It follows from \eqref{solution_pell_final_form} and \eqref{pell_multiple} that
\begin{gather}
		X_n =XT_n+DYU_n=\frac{1}{2}\big[\big(X+Y\sqrt{D}\big)A^n+\big(X-Y\sqrt{D}\big)A^{-n}\big],\nonumber\\
		Y_n =XT_n+DYU_n=\frac{1}{2\sqrt{D}}\big[\big(X+Y\sqrt{D}\big)A^n-\big(X-Y\sqrt{D}\big)A^{-n}\big],\label{generalized_Pell_solution_format}
	\end{gather}
where we denoted $ A=t+u\sqrt{D}>1 $.

\subsection{Quadratic Diophantine equation in two variables: revisited}
Now we circle back to the solutions of our Diophantine equation \eqref{general_diophantine_equation}. We have already seen that in the case when $ N\neq 0 $ and $ D>0 $ is not a perfect square, the solutions of \eqref{general_diophantine_equation} is related to those of generalized Pell's equation \eqref{generalized_pell_equation}. Once we obtain all solutions of \eqref{generalized_pell_equation}, we need to identify for each $ (X,Y)\in \mathcal{S} $, those values of $ n\in \mathbb{Z} $ such that
\begin{equation*}
	\begin{cases}
		X_n\equiv E \quad \bmod  D,\\
		Y_n\equiv b(X_n-E)/D+d \quad \bmod 2a,
	\end{cases}
\end{equation*}
or equivalently,
\begin{equation*}
	\begin{cases}
		X_n\equiv E\quad \bmod  D,\\
		DY_n\equiv bX_n-bE+Dd\quad \bmod 2aD.
	\end{cases}
\end{equation*}
The efficient way \cite{sawilla} is to check a necessary condition
\begin{equation}\label{necessary_requirement_diophantine}
	2a \mid X-bY.
\end{equation}
If \eqref{necessary_requirement_diophantine} and
\begin{equation*}
	D~\left|~\frac{dD-bE-DY+bX}{2a}\right.
\end{equation*}
hold, then \eqref{generalized_Pell_solution_format} produces solutions for \eqref{general_diophantine_equation} for all even $ n $. If \eqref{necessary_requirement_diophantine} and
\begin{equation*}
	D~\left|~\frac{dD-bE-DYt+bXt}{2a}\right.
\end{equation*}
hold, then \eqref{generalized_Pell_solution_format} produces solutions for \eqref{general_diophantine_equation} for all odd $ n $. If none of these holds, then there is no solution of \eqref{general_diophantine_equation}.

\subsection*{Acknowledgements}
I would like to thank my advisor Pavel Putrov not only for introducing me this interesting topic but also for his continuous support and guidance. I would also like to thank the anonymous referees who provided insightful and detailed comments and suggestions on a earlier version of the paper.

\pdfbookmark[1]{References}{ref}
\LastPageEnding

\end{document}